\documentclass[12pt,a4paper,oneside]{article}
\usepackage{amsopn}
\usepackage{amsfonts,url}
\usepackage{color}
\usepackage[margin=23mm]{geometry}
\usepackage{amssymb, amsmath, amsthm, amscd, accents, mathtools, mathrsfs}
\usepackage{dsfont}
\usepackage{graphicx,times,bbm,url}
\usepackage[T1]{fontenc}
\usepackage{hyperref}
\usepackage{comment}
\usepackage{tabularx}
\usepackage{array}
\usepackage{enumitem}
\usepackage{float}

\title{Asymptotics of transition densities for L\'evy processes with stochastic resetting}

\author{Kacper Budnik, Tomasz Grzywny, Pawe{\l} Sztonyk\\
Faculty of Pure and Applied Mathematics,\\
Wroc\l aw University of Science and Technology,\\
Wyb. Wyspia\'nskiego 27, 50-370 Wroc\l aw, Poland}

\date{\today}

\newtheorem{lemma}{Lemma}
\numberwithin{lemma}{section}
\newtheorem{theorem}[lemma]{Theorem}
\newtheorem{corollary}[lemma]{Corollary}

\theoremstyle{definition}
\newtheorem{remark}[lemma]{Remark}

\newcommand{\dx}[1][x]{\text{d}#1}

\newcommand*\mP{\mathbb{P}}
\newcommand*\mE{\mathbb{E}}
\newcommand{\Rd}{\mathbb{R}^d}
\newcommand{\R}{\mathbb{R}}

\numberwithin{equation}{section}

\begin{document}
\maketitle

\footnotetext{ The study has been supported by Grant No.~2023/49/B/ST1/00678 of the National Science Center (Poland). keywords: Stochastic processes, Stochastic Resetting, L\'evy processes, Asymptotics, Probability Density Functions, Probability, Renewal Theory, Stochastic Analysis, Transition Densities
}

\emph{e-mail address: kacper.budnik@pwr.edu.pl; tomasz.grzywny@pwr.edu.pl; pawel.sztonyk@pwr.edu.pl}

\begin{abstract}
    We derive an explicit representation for the distribution of a Lévy process with total stochastic resetting and use it to investigate the asymptotic behaviour of the corresponding transition densities. The proposed approach applies to several important classes of Lévy processes, including relativistic stable processes, and yields new asymptotic estimates for their transition densities.
\end{abstract}

\counterwithout{equation}{section}
\numberwithin{equation}{section}
\numberwithin{figure}{section}

%\tableofcontents

\section{Introduction and preliminaries}
Stochastic resetting has recently become an active area of research due to its broad range of applications in physics, biology, computer science and finance. In its simplest form, the underlying process is restarted at random times, typically governed by a Poisson process. Such resetting mechanisms substantially modify the long-time behaviour of the original process and often produce stationary distributions even when the underlying process is transient.

While stochastic resetting has been extensively investigated for Brownian motion and diffusion processes, much less is known for general Lévy processes, especially concerning precise asymptotic behaviour of their transition densities. The main objective of this paper is to establish such asymptotic results for a broad class of Lévy processes under total resetting.

Throughout this paper we study the stochastic process $\boldsymbol{X}$ obtained from a L\'evy process $\boldsymbol{Y}$ in $\Rd$ by (total) resetting, i.e., $\boldsymbol{Y}$ is reset to the origin at exponential random times. 

Both the present model and the more general framework of partial resetting have attracted considerable attention in recent years (see \cite{Grzywny2024, kolb} or \cite{Evans2020,Gupta2022} for surveys).
They arise naturally in applications to physics (see, e.g., \cite{EvansMajumdar2011}), biology (\cite{RoldanLisicaGrill2016}), computer science (\cite{Dumas2002,MontanariZecchina2002,Guillemin2002}) and finance (\cite{Kou2002,Mendoza2010,MERTON1976}).

Let $d\in \{ 1,2,...\}$ and $(\Omega, \mathcal{F}, \mP)$ be a probability space. We consider a L\'evy process $\boldsymbol{Y} = (Y_t, \Omega, \mathcal{F}, \mP)$ on $\Rd$, i.e., a stochastic process with 
stationary and independent increments, continuous in probability and determined by its characteristic function
\[
  \mE e^{i \langle \xi, Y_t \rangle} = e^{-t\psi(\xi)}, \quad x \in\Rd, t\geq 0,
\]
where $\psi$ is given by the L\'evy--Khintchine formula
\[
  \psi(\xi) = -i\langle \gamma,\xi\rangle + \frac12 \langle \xi, A\xi\rangle
      +\int_{\Rd} \left(1 - e^{i\langle\xi,y\rangle}+i\langle \xi,y\rangle1_{(0,1)}(|y|)\right)\, \nu(\dx[y]), \quad \xi\in\Rd.
\]
Here $\gamma\in\Rd,$ $A$ is a symmetric nonnegative--definite $d\times d$ matrix and $\nu$ is
a measure on $\Rd\setminus\{ 0 \}$ satisfying $\int_{\Rd\setminus\{ 0 \}} (|y|^2\wedge 1)\, \nu(\dx[y]) < \infty$ (the \emph{L{\'e}vy measure} of $\boldsymbol{Y}$).

Let $\boldsymbol{N}$ be a Poisson process with intensity $\lambda>0,$ 
independent of $\boldsymbol{Y}.$ The process with resetting $\boldsymbol{X}$ corresponding to $\boldsymbol{Y}$ is defined as
     \begin{equation*}\label{eq:def_procX}
        X_t=\left\{\begin{array}{lll}
            Y_t             &\text{for } & t<T_1, \\
            Y_t-Y_{T_n}     &\text{for } & T_n\leq t < T_{n+1},
        \end{array}\right.
    \end{equation*}
where $\left\{T_i, i\in\mathbb{N}_0\right\}$ %of $\boldsymbol{N}$ 
are the Poisson arrival moments of $\boldsymbol{N}$.  
Observe that, for every $\omega\in\Omega$ a path $t\mapsto X_t(\omega)$ is constructed solely from the trajectory  $t\mapsto Y_t(\omega)$. This is possible because we have limited ourselves to L\'{e}vy processes.

In the present paper, we focus on the asymptotic behavior of densities of the processes
with total resetting as both time and space variables tend to infinity. Our starting point is an explicit representation formula for the one-dimensional distributions of the reset process. This representation, established in Theorem \ref{th:formula}, serves as the main tool for deriving the asymptotic behaviour of the corresponding transition densities. The resulting one-dimensional distributions agree with that used in the literature, which is derived from the recursive formula. Using this, we derive the asymptotics of its density for several important examples of L\'evy processes. A particularly interesting example is provided by relativistic stable processes (see Theorem \ref{th:relatproc}). Let us mention that recently an analogous formula for the density was obtained in \cite{Grzywny2024} (see also \cite{dibello2025}) for stable processes with partial resetting. Restricting ourselves to total resetting allows us to obtain substantially more general asymptotic results.

For $\boldsymbol{X}$, $\boldsymbol{Y}$, $\boldsymbol{N}$ defined as above, we derive the following formula for the one--dimensional distributions of $\boldsymbol{X}.$
\begin{theorem}\label{th:formula} For all $\lambda>0, t>0$, $x\in\Rd$ and Borel sets $A\subset\Rd$ we have
  \begin{equation}\label{eq:res_distrib}
    \mP^x(X_t\in A) = e^{-\lambda t} \mP^x(Y_t\in A)  + \int_0^t \lambda e^{-\lambda s} \mP^0(Y_{s}\in A)\, \dx[s].
  \end{equation}
  In particular, if for every $t>0$ the random variable $Y_t$ is  absolutely continuous with respect to the Lebesgue measure with a corresponding density $p(t,\cdot)$ then $X_t$ is also absolutely continuous for every $t>0$ with a density $p_\lambda(t,x,y)$ given by the formula
  \begin{equation}\label{eq:res_den_nores}
        p_\lambda(t,x,y)=e^{-\lambda t} p(t,y-x) + \int\limits_{0}^t\lambda e^{-\lambda  s}p(s,y)\,\dx[s], \quad t>0, x,\,y\in \Rd.
    \end{equation}
\end{theorem}

\begin{proof}
Proceeding in the same way as in the proof of \cite[(3.1)]{Grzywny2024}, we obtain for all $t>0$ and Borel sets $A\subset\Rd$ the formula
\begin{equation}\label{eq:rec}
    \begin{split}
\mP^x(X_t\in A) &= e^{-\lambda t} \mP^x(Y_t\in A)  + \int_0^t \lambda e^{-\lambda s} \mP^0(X_{t-s}\in A)\, \dx[s]\\
        &  = e^{-\lambda t} \mP^x(Y_t\in A)  + e^{-\lambda t} \int_0^t \lambda e^{\lambda s} \mP^0(X_{s}\in A)\, \dx[s].
    \end{split}
\end{equation}
Iterating \eqref{eq:rec} we obtain
\begin{align*}
    \mP^x(X_t\in A) & = 
    e^{-\lambda t} \mP^x(Y_t\in A)  + \lambda e^{-\lambda t} \int_0^t \mP^0(Y_s \in A)\, \dx[s] + \lambda^2 e^{-\lambda t} \int_0^t \int_0^s e^{\lambda u} \mP^0(X_u \in A) \dx[u] \dx[s]\\
    & = e^{-\lambda t} \mP^x(Y_t\in A)  + \sum_{k=1}^n F_k(t, A) + R_{n+1}(t,A),
\end{align*}
where 
\[
    F_k(t,A) = \lambda^k e^{-\lambda t} \idotsint_{0<t_1<t_2<\ldots<t_k<t} \mP^0(Y_{t_1}\in A)\, \dx[t_1]...\dx[t_k]
\]
and
\[
    R_{n+1}(t,A) = \lambda^{n+1} e^{-\lambda t} \idotsint_{0<t_1<t_2<\ldots<t_{n+1}<t} e^{\lambda t_1}\mP^0(X_{t_1}\in A)\, \dx[t_1]...\dx[t_{n+1}].
\]
Using Cauchy's formula for repeated integration we get
\[
   F_k(t,A) = e^{-\lambda t} \frac{\lambda^k}{(k-1)!} \int_0^t (t-s)^{k-1} \mP^0(Y_{s}\in A) \dx[s]
\]
and
\[
   R_{n+1}(t,A) =  e^{-\lambda t} \frac{\lambda^{n+1}}{n!} \int_0^t (t-s)^{n} e^{\lambda s} \mP^0(X_{s}\in A)\, \dx[s]
   \leq  \frac{\lambda^{n+1} t^{n+1}}{n!},
\]
hence $\lim_{n\to\infty} R_{n+1}(t,A) = 0$ and
\begin{align*}
    \mP^x(X_t\in A) & = e^{-\lambda t} \mP^x(Y_t\in A)  + e^{-\lambda t} \sum_{n=1}^\infty \frac{\lambda^n}{(n-1)!} \int_0^t (t-s)^{n-1}\mP^0(Y_s\in A)\, \dx[s] \\
    & = e^{-\lambda t} \mP^x(Y_t\in A)  + \lambda \int_0^t e^{-\lambda s} \mP^0(Y_s\in A)\, \dx[s]
\end{align*}
which is \eqref{eq:res_distrib}. The formula \eqref{eq:res_den_nores} is an obvious conclusion.
\end{proof}
In the same way as in the proof of [11, (3.1)] one can prove that $\boldsymbol{X}$ with respect to natural filtration is a Markov process with c\'adl\'ag paths.  However,  \eqref{eq:res_distrib} shows that the process with resetting is not a Feller process. Furthermore, the transition operators $P_tf(x)=\mE^x f(X_t)$ cannot be defined on $L^p$ space for any $p\in[1,\infty)$.

\begin{remark}An immediate consequence of the above theorem is the fact  that $X_t$ converges in distribution to $X_\infty\sim \lambda U_\lambda$, where $U_\lambda$ is $\lambda$-potential measure of $\boldsymbol{Y}$, that is $U_\lambda(A)=\int^\infty_0e^{-
\lambda s}\mP^0(X_s\in A)\dx[s]$. Furthermore, if the distribution of $Y_t$ is absolutely continuous for every $t>0$ and $Y_{t_0}$ has bounded density for some $t_0>0$, \begin{equation*}\label{eq:sup_ergodic}\lim_{t\to\infty}\sup_{x,y\in\Rd}|p_\lambda(t,x,y)-\lambda u_\lambda(y)|=0,\end{equation*}
where $u_\lambda$ is a density of $U_\lambda$. Indeed, for $t\geq t_0$ we have $\sup_{z\in\Rd}p(t,z)\leq \sup_{z\in\Rd}p(t_0,z)=:c$ and therefore
$$\sup_{x,y\in\Rd}|p_\lambda(t,x,y)-\lambda u_\lambda(y)|\leq e^{-\lambda t}c+\lambda\int^\infty_te^{-\lambda s}\,c\,\dx[s].$$
Let $A\subset \Rd$. If we additionally assume that $\inf_{y\in A}u_\lambda(y)>0$, then we obtain also 
\begin{equation*}%\label{eq:sup_ergodic}
\lim_{t\to\infty}\sup_{x\in\Rd,\,y\in A}\left|\frac{p_\lambda(t,x,y)}{\lambda u_\lambda(y)}-1\right|=0.\end{equation*}
\end{remark}
The above assumptions are usually satisfied for compact sets $A$, therefore we will examine the behaviour of density $p_\lambda$ for large $t$ and $|y|$.

In what follows, we consider several examples of L\'evy processes and, 
using Theorem \ref{th:formula}, obtain the asymptotic behavior 
of the corresponding resetting densities. Below we use the
notation $A(t,y)\sim B(t,y)$ to indicate that $\lim\limits_{\substack{|y|\to\infty\\ t\to \infty}}\frac{A(t,y)}{B(t,y)}=1.$ 

First of all, we consider the example of Brownian motion $Y_t=B_t$ in $\Rd$ with
the characteristic exponent $\psi(\xi)=\frac12 |\xi|^2$ and the densities
$p(t,y) = (2\pi t)^{-d/2}e^{-\frac{|y|^2}{2t}}, t>0, y\in\Rd.$ The density of
the corresponding process with resetting is for every $\lambda>0$ and $t>0$, $x,\,y\in\Rd$, given by
\begin{align*}
  p_\lambda(t,x,y) & = e^{-\lambda t} (2\pi t)^{-d/2}e^{-\frac{|y-x|^2}{2t}} +
    \int_0^t \lambda e^{-\lambda s} (2\pi s)^{-d/2}e^{-\frac{|y|^2}{2s}} \, \dx[s], \\
    & = (2\pi t)^{-d/2}e^{-\frac{|y-x|^2}{2t}-\lambda t} + 
    \lambda (2\pi)^{-d/2}|y|^{1-\frac{d}{2}} \int_0^{t/|y|} u^{-d/2}e^{-|y|\left(\lambda u + \frac{1}{2u}\right)}\, \dx[u].
\end{align*}
With little effort, we obtain the following description of its asymptotics. It depends on the ratio $\frac{t}{|y|}$ and
changes its character on the threshold $\frac{t}{|y|}=\frac{1}{\sqrt{2\lambda}}.$ For $\frac{t}{|y|}\geq M > \frac{1}{\sqrt{2\lambda}}$ the obtained asymptotic coincides with 
the resolvent kernel ($\lambda$-potential) $u_\lambda(y) = \frac{1}{\pi^{d/2}}\left(\frac{\sqrt{2\lambda}}{2|y|}\right)^{\frac{d}{2}-1}K_{\frac{d}{2}-1}(\sqrt{2\lambda}\,|y|)$ (see \cite[Section 7.3]{SchillingPartzsch2014}, here $K_\mu$ denotes the modified Bessel function of the second kind). 
\begin{theorem}\label{th:Wiener}
    Let $\boldsymbol{Y}$ be the Brownian motion in $\Rd.$ Then
    for all $\lambda>0, t>0$ the density $p_\lambda$ of the
    corresponding process $\boldsymbol{X}$ with resetting exists and, for any $x\in\Rd$,
    \begin{enumerate}
        \item For every $M>\frac{1}{\sqrt{2\lambda}}$ we have
        \begin{equation}\label{eq:WienerFirst}
          p_\lambda(t,x,y) \sim \frac{\lambda^{\frac{d+1}{4}}}{2^{\frac{d+1}{4}}\pi^{\frac{d-1}{2}}} \frac{e^{-\sqrt{2\lambda}|y|}}{|y|^{\frac{d-1}{2}}}, \quad \text{ on } \quad \frac{t}{|y|}\geq M.
        \end{equation}
        \item For every $0<  M <\frac{1}{\sqrt{2\lambda}}$ we have
        \begin{equation}\label{eq:WienerSecond}
            p_\lambda(t,x,y) \sim e^{-\lambda t}p(t,y) \left(\frac{2\lambda}{(|y|/t)^2-2\lambda}+e^{\frac{-|x|^2+2\langle x,y
            \rangle}{2t}}\right) %=  \frac{(2\pi t)^{-d/2}e^{-\frac{|x|^2}{2t}-\lambda t}(|x|/t)^2}{(|x|/t)^2-2\lambda}
            , \quad \text{ on } \quad 0< \frac{t}{|y|} \leq M.
        \end{equation}
        \item For $\frac{t}{|y|}=\frac{1}{\sqrt{2\lambda}}$ we get
        \[
             p_\lambda(t,x,y) \sim \frac{\lambda^{\frac{d+1}{4}}}{2\cdot 2^{\frac{d+1}{4}}\pi^{\frac{d-1}{2}}} \frac{e^{-\sqrt{2\lambda}|y|}}{|y|^{\frac{d-1}{2}}}.
        \]
    \end{enumerate}
\end{theorem}
We prove Theorem \ref{th:Wiener} in Section \ref{section:Wiener}. We also highlight here a particularly interesting and challenging example of the relativistic stable processes, i.e., the L\'evy process
    with the characteristic exponent:
    \[
       \psi(\xi) = \left(m^{2/\alpha} + |\xi|^2\right)^{\alpha/2} - m, \quad \xi\in\Rd,
    \] 
        where $\alpha\in(0,2)$, $m > 0$. They form a~significant class of processes corresponding to the so-called relativistic stable operators
\[
    L = m - \left(m^{2/\alpha} - \Delta\right)^{\alpha/2}.
\]
 Especially for $\alpha=1,$ this operator plays a very important role in relativistic quantum mechanics as it corresponds to
the kinetic energy of a relativistic particle with mass $m$ (see, e.g., \cite{Carmona1989,Froehlich2007,Herr2014,Lieb2010} and the references therein).

Let $\eta:\: (0,\infty) \to (0,\infty)$ be the density of the $\alpha/2$--stable subordinator, i.e., $\int_0^\infty e^{-su} \eta(t,u)\, \dx[u] = e^{-ts^{\alpha/2}}, \, s>0.$

\begin{theorem}\label{th:relatproc}
    Let $\alpha\in (0,2), m>0$ and $\boldsymbol{Y}$ be the relativistic stable process.
    Then the density $p_\lambda$ of the corresponding process $\boldsymbol{X}$ with resetting exists for all $\lambda>0, t>0$ and,  for any $x\in\Rd$,:
    \begin{enumerate}
        \item If $\lambda > m $ then
        \begin{align*}
          p_\lambda(t,x,y) 
          &\sim  \frac{\lambda\alpha m^{\frac{d+\alpha-1}{2\alpha}}}{2^{\frac{d-\alpha+1}{2}}\pi^{\frac{d-1}{2}}\Gamma(1-\frac{\alpha}{2})(\lambda-m)^2} \frac{e^{-m^{1/\alpha}|y|}}{|y|^{\frac{d+\alpha+1}{2}}} .
    \end{align*}
    \item If $\lambda = m$  then
        \begin{equation*}
        p_\lambda(t,x,y) \sim \frac{m^{\frac{d+\alpha-1}{2\alpha}}e^{-m^{1/\alpha}|y|}}{2^{\frac{d+\alpha-1}{2}}\pi^{\frac{d-1}{2}} |y|^{\frac{d-\alpha+1}{2}}}
        \int_0^{\frac{\sqrt{2^\alpha m}}{|y|^{\alpha/2}}t}\eta(s,1)\, \dx[s].
    \end{equation*}
    
    \item If $\lambda < m$ and $\frac{t}{|x|^{\alpha/2}}<M$ for some constant $M>0$ then
       \begin{equation*}
        \begin{split}
            p_\lambda(t,x,y) & \sim 
            \left(\lambda + e^{m^{1/\alpha}\langle \frac{y}{|y|},x\rangle }\frac{\eta\left(\frac{t}{|y-x|^{\alpha/2}}\sqrt{2^\alpha m},1\right) (m-\lambda)}{\eta\left(\frac{t}{|y|^{\alpha/2}}\sqrt{2^\alpha m},1\right)}\right) \\
            &\times \frac{ m^{\frac{d-\alpha-1}{2\alpha}}e^{-m^{1/\alpha}|y|}}{2^{\frac{d+\alpha-1}{2}}\pi^{\frac{d-1}{2}}|y|^{\frac{d-\alpha+1}{2}}}
        \int_0^{\frac{\sqrt{2^\alpha m}}{|y|^{\alpha/2}} t} e^{(m-\lambda)\frac{|y|^{\alpha/2}}{\sqrt{2^\alpha m}}s} \eta(s,1)\, \dx[s].
        \end{split}
    \end{equation*}
    \end{enumerate}
    
\end{theorem}

We give the proof of Theorem \ref{th:relatproc} in Section \ref{section:RelStP}. 
Note that the behavior of $p_\lambda$ depends
here strongly on the value of the parameter $m.$
For small parameter $m$ (smaller than $\lambda$) the behavior of density does not depend on rate of convergence of time and space. It behaves like the density of ergodic measure ($\lambda u_\lambda$). If $m=\lambda$ and $t^2/|y|^\alpha$ goes to infinity the behavior is the same, but if this ratio is small $p_\lambda$ behaves like $c \,t^2u_\lambda$. 

The reader may be interested to know that an analogous phenomenon appears in studies of the $\lambda$-potential kernel $u_\lambda(y) =  \lim_{t\to\infty} (p_\lambda(t,0,y) - e^{-\lambda t}p(t,y))/\lambda$ of the relativistic process. Namely, it was recently proved in \cite{KaletaSchillingSztonyk2025} that $u_\lambda$ for $\lambda>m$ is for large $y$ comparable to the density of the L\'evy measure
of $\boldsymbol{Y}$ (see \ref{eq:relLevy}) but has a different asymptotic for $\lambda\leq m,$ i.e., for every $\varepsilon>0$ there
exists $c>0$ such that
\[
  c^{-1} e^{-\left(\sqrt{m^{2/\alpha}-(m-\lambda)^{2/\alpha}}+\varepsilon\right)|y|} \leq 
  u_\lambda(y) \leq c e^{-\left(\sqrt{m^{2/\alpha}-(m-\lambda)^{2/\alpha}}-\varepsilon\right)|y|},\quad |y|\geq 1.
\]

\section{Brownian motion}\label{section:Wiener}
In this section, we provide the proof of Theorem \ref{th:Wiener}.
\begin{proof}[Proof of Theorem \ref{th:Wiener}.]
    Assume first that $\frac{t}{|y|} \geq M > \frac{1}{\sqrt{2\lambda}}$ and denote $f(u)=\lambda u + \frac{1}{2u}, \, u\in (0,\infty).$ The function $f$ is decreasing
    on $\left(0,\frac{1}{\sqrt{2\lambda}}\right]$ and increasing on $\left[\frac{1}{\sqrt{2\lambda}},\infty\right)$ with the global minimum
    at $u_0=\frac{1}{\sqrt{2\lambda}}, f(u_0) = \sqrt{2\lambda}.$ We observe that 
    \[
      \int_0^{\frac{t}{|y|}} u^{-d/2} e^{-|y|f(u)}\, \dx[u] \sim \int_0^M u^{-d/2} e^{-|y|f(u)}\, \dx[u],
    \]
    and by Laplace method (see Lemma \ref{thm:laplace}) we get
    \eqref{eq:WienerFirst}. 

Let $ \frac{t}{|y|}\leq M < \frac{1}{\sqrt{2\lambda}}$ 
%we first observe that
%\[
 % \int_0^{\frac{t}{|x|}} u^{-d/2} e^{-|x|f(u)}\, \dx[u]
  %\sim \int_{m/2}^{\frac{t}{|x|}} u^{-d/2} e^{-|x|f(u)}\, %\dx[u].
%\]
We denote $\kappa :=t/|y|$ and by the substitution
$w=f(u) - f(\kappa),$ $u(w)=\frac{w+f(\kappa) - \sqrt{(w+f(\kappa))^2 - 2\lambda}}{2\lambda}= \frac{1}{w+f(\kappa) + \sqrt{(w+f(\kappa))^2 - 2\lambda}}, u'(w) = \frac{-u(w)}{\sqrt{(w+f(\kappa))^2-2\lambda}},$ and for $h(s) = \frac{(s + \sqrt{s^2 - 2\lambda})^{d/2-1}}{\sqrt{s^2-2\lambda}}$, 
we obtain
\begin{align*}
    \int_{0}^{\frac{t}{|y|}} u^{-d/2} e^{-|y|f(u)}\, \dx[u]  & = e^{-|y|f(\kappa)} \int_0^{\infty} (u(w))^{-d/2} e^{-|y|w} (-u'(w))\, \dx[w] \\
    & =e^{-|y|f(\kappa)} \int_0^{\infty} e^{-|y|w}h(w+f(\kappa))\, \dx[w], \\
    & = e^{-|y|f(\kappa)} \int_0^{\infty} e^{-|y|w}h(f(\kappa))\, \dx[w] \\
    & \quad +\,\, e^{-|y|f(\kappa)} \int_0^{\infty} e^{-|y|w}\int_0^w h'(s+f(\kappa))\, \dx[s]\, \dx[w] \\
    &= \frac{e^{-|y|f(\kappa)}}{|y|}\left[h(f(\kappa))+\int_0^{\infty} e^{-w}\int_0^{w/|y|} h'(s+f(\kappa))\, \dx[s]\, \dx[w]\right].
\end{align*}
We observe that $h'(u) = h(u)\frac{(d/2-1)\sqrt{u^2-2\lambda}-u}{u^2-2\lambda},$ so there exist $c_1=c_1(M,d), c_2=c_2(M,d)$ such that
$|h'(u)| \leq c_1 \frac{h(u)}{u} \approx u^{d/2-3}$ for $u\in [f(M),\infty)$ and\\ 
$|h'(s+f(\kappa))| \leq c_2 h(\max\{s,f(\kappa)\})/\max\{s,f(\kappa)\}$, for $s\geq 0$. Hence

\begin{align*}
    &\left| \int_0^{\infty} e^{-w}\int_0^{w/|y|} h'(s+f(\kappa))\, \dx[s]\, \dx[w] \right|  \\
    &\leq
    \frac{c_3 h(f(\kappa))}{|y|}\int^{\infty}_0e^{-w}w\,\dx[w] 
+ c_3\int^\infty_{|y|f(\kappa)}e^{-w}\int^{w/|y|}_{f(\kappa)}\frac{h(s)}{s}\dx[s]\, \dx[w]\\
&\leq  \frac{c_4 h(f(\kappa))}{|y|} + \frac{c_5}{f(\kappa)}h(f(\kappa))\int^\infty_{|y|f(\kappa)}e^{-w}\int^{w/|y|}_{f(\kappa)}\left(\frac{s}{f(\kappa)}\right)^{(d-6)/2}\dx[s]\, \dx[w]
\end{align*}
%where $$\mathrm{I}=$$ 
%Obviously, $h'(s)$ is continuous and therefore bounded on $[f(M),f(m/2)],$ and denoting $c_0 = \sup\{|h'(s)|:\: s\in [f(M),f(m/2)]\},$ we get
%We have
%\[
%  h'(s) = \frac{(s + \sqrt{s^2 - 2\lambda})^{d/2-1}\left(\left(\frac{d}{2}-1\right)\sqrt{s^2-2\lambda}\,  - s\right)}{(s^2 - 2\lambda)^{3/2}}
%\]
%\[
%  \int_0^{f(m/2)-f(\kappa)} e^{-|x|w}\int_0^w \left| h'(s+f(\kappa))\right|\, \dx[w]
 %  \leq c_0 \int_0^{f(m/2)-f(\kappa)} we^{-|x|w}\, \dx[w]
  % \leq \frac{c_0}{|x|^2}, %\left(\frac{1}{|x|^2}- \frac{e^{-c_1|x|}}{|x|}\left(c_1+\frac{1}{|x|}\right)\right),
%\]
and this yields
\begin{align*}
  \int_{0}^{\frac{t}{|y|}} u^{-d/2} e^{-|y|f(u)}\, \dx[u] & %\sim e^{-|y|f(\kappa)}h(f(\kappa))\frac{1}{|y|}
  \sim \frac{e^{-|y|f(\kappa)}h(f(\kappa))}{|y|} \\
  & = \frac{2e^{-\lambda t-\frac{|y|^2}{2t}} (|y|/t)^{d/2}}{|y|\left((|y|/t)^2 - 2\lambda\right)} ,
\end{align*}
%since $f(m/2)-f(\kappa)\geq f(m/2)-f(m),$ 
and \eqref{eq:WienerSecond} follows by \eqref{eq:res_den_nores}.

The case $\frac{t}{|y|}=\frac{1}{\sqrt{2\lambda}}$ follows by an obvious modification of the original 
Laplace's method (see, e.g., \cite{Nemes2013}). The details are omitted.
\end{proof}

\section{Processes with heavy tails}
\setcounter{definition}{0}

First, we prove the following lemma, which allows us to describe the asymptotic behavior of
the process $\boldsymbol{X}$ in several important cases below.
\begin{lemma}\label{lemma:asymptotic}
    Assume that for every $t>0$ there exists a transition density $p(t,y)$ of the L\'evy process $\boldsymbol{Y}$  and let $\theta\in\mathbb{S}^{d-1},$ $\lambda>0$, and $x\in\Rd$. If there exist Borel functions $f,A,\tilde{A}:\mathbb{R}_+\to\mathbb{R}_+$, and $r_0>0$ such that
	\[
        \lim_{r\to\infty} \frac{p(t,r\theta)}{f(r)}= A(t), \, t>0,
        \quad\text{ and } \quad
        \frac{p(t,r\theta-z)}{f(r)}\leq \tilde{A}(t),\, t>0, r>r_0, z\in \{0,x\}, 
	\]
	\[
        \lim\limits_{t\to\infty} e^{-\lambda t} \tilde A(t)=0
        \quad \text{ and } \quad
        \int\limits_{0}^\infty e^{-\lambda t} \tilde A(t)\, \dx[t]<\infty,
	\]
	then the density $p_\lambda$ of the corresponding process with resetting $\boldsymbol{X}$ satisfies
	\[
		\lim\limits_{\substack{r\to\infty\\t\to\infty}}\frac{p_\lambda(t,x,r\theta)}{f(r)}=\int\limits_0^\infty \lambda e^{-\lambda s}A(s)\, \dx[s].
	\]
\end{lemma}

\begin{proof}
We have
\[
  0 \leq \frac{e^{-\lambda t}p(t,r\theta-x)}{f(r)}  
  \leq e^{-\lambda t} \tilde A(t)  \xrightarrow{t \to \infty} 0, \quad r>r_0.
\]
Furthermore,
\[
   \lambda e^{-\lambda s} \frac{p(s,r\theta)}{f(r)} \leq \lambda e^{-\lambda s} \tilde{A}(s), \quad s>0, r>r_0,
\]
and $\int_0^\infty e^{-\lambda s}\tilde{A}(s)\, \dx[s] <\infty,$
%\[
%    \int\limits_0^t \lambda e^{-\lambda s}\frac{p(s,r\theta)}{f(r)}\dx[s]\leq\lambda\int\limits_0^\infty e^{-\lambda s}\frac{p(s,r\theta)}{f(r)}\dx[s]\leq\lambda\int\limits_0^\infty e^{-\lambda s}\tilde A(s)\dx[s]<\infty, \quad t>0, r>r_0,
%\] 
hence, using the dominated convergence, we obtain
\[
    \lim\limits_{\substack{r\to\infty\\t\to\infty}	}\frac{1}{f(r)}\int\limits_0^t\lambda e^{-\lambda s}p(s,r\theta)\dx[s]
    =\int_0^\infty \lim\limits_{\substack{r\to\infty\\t\to\infty}}\lambda e^{-\lambda s}\frac{p(s,r\theta)}{f(r)}\mathds{1}_{(0,t)}\dx[s]
    =\int\limits_0^\infty\lambda e^{-\lambda s}A(s)\dx[s],
\]
and by \eqref{eq:res_den_nores} the lemma follows.

\end{proof}

\subsection{Isotropic stable process.}
\begin{theorem}
    Let $\alpha\in (0,2)$ and $Y_t$ be an isotropic $\alpha$--stable process in $\Rd$, i.e., 
    \[
        \mE e^{i \langle\xi,Y_t\rangle } = e^{-t|\xi|^\alpha}, \quad \xi\in\Rd, t>0.
    \]
    Then the corresponding process with resetting $X_t$ for all $\lambda>0, x\in\Rd, t>0$ has a density $p_\lambda(t,x,y)$
    such that
    \begin{equation}\label{eq:asstabres}
    p_\lambda(t,x,y)\sim\frac{\mathcal{A}_{d,\alpha}}{\lambda}|y|^{-(d+\alpha)},
	\end{equation}
    where
    \begin{equation}\label{eq:lim_stable}
		\mathcal{A}_{d,\alpha}=\alpha2^{\alpha-1}\pi^{-d/2-1}\sin\left(\frac{\alpha\pi}{2}\right)\Gamma\left(\frac{\alpha}{2}\right)\Gamma\left(\frac{\alpha+d}{2}\right).
    \end{equation}
\end{theorem}
\begin{proof}
The process $Y_t$ has for every $t>0$ a transition density $p(t,z)$ which satisfies
the scaling property
    \[
        p(t,z) = t^{-d/\alpha} p(1,z/t^{1/\alpha}), \quad z\in\Rd, t>0.
    \]
    It was proved in \cite{BlumenthalGetoor1960} that
    \begin{equation}\label{eq:BlumenthalGetoor}
		\lim\limits_{|z|\to\infty}|z|^{d+\alpha}p(1,z)=\mathcal{A}_{d,\alpha},%\quad \text{and}\quad p(t,x)=t^{-d/\alpha}p_\alpha(1,t^{-1/\alpha}x),
    \end{equation}	
%	Since $|t^{-1/\alpha}x|\to\infty$ for each $t>0$ and self-similarity of stable process holds
	%Since $|t^{-1/\alpha}x|\to\infty$ as $|x|\to\infty$ and that stable processes are self-similar, holds
        hence %for every $x\in\Rd$ we have
	\[%\label{eq:lim_alpha}
            \lim\limits_{|y|\to\infty} \frac{p(t,y)}{|y|^{-d-\alpha}} =
            \lim\limits_{|y|\to\infty} \frac{ p(1,y/t^{1/\alpha})}{|y/t^{1/\alpha}|^{-d-\alpha}}
		t = \mathcal{A}_{d,\alpha} t, \quad t>0.
	\]
        Moreover, we have 
        $p(1,z) = (2\pi)^{-d}\int_{\Rd} e^{-i\langle z,\xi\rangle}e^{-|\xi|^\alpha}\, \dx[\xi]\leq 
        (2\pi)^{-d}\int_{\Rd} e^{-|\xi|^\alpha}\, \dx[\xi] $  for all $z\in\Rd,$ 
        and this, \eqref{eq:BlumenthalGetoor} and the scaling property yield %Furthermore, from Blumenthal i Getoor[TBA], it's known that
	\[
		p(t,y-x)\leq c_1 t|y|^{-d-\alpha},\quad  x,y\in\Rd, |y|>2|x|, t>0,
	\]
        for some constant $c_1>0.$
	Thus the density $p(t,z)$ of the process $Y_t$ satisfies the assumptions of Lemma~\ref{lemma:asymptotic} for $f(r)=r^{-d-\alpha}$, $A(t)=\mathcal{A}_{d,\alpha}t$ and $\tilde{A}(t)=c_1 t$, and by this lemma and the isotropy of $p(t,\cdot),$
    we obtain 
    %$p_\lambda (t,0,y) \sim \frac{\mathcal{A}_{d,\alpha}}{\lambda}|y|^{-(d+\alpha)}.$ 
    %Using this and the fact
    %that 
    %\[
    %\frac{|p_\lambda(t,x,y)-p_\lambda(t,0,y)|}{|y|^{-d-\alpha}} = 
    %\frac{e^{-\lambda t}|p(t,y-x)-p(t,y)|}{|y|^{-d-%\alpha}} \leq c_2 t e^{-\lambda t} \to 0,\]
    %as $t,|y| \to\infty,$ we get
    \eqref{eq:asstabres}.
    \end{proof}
\subsection{Stable process with independent coordinates}
\begin{theorem}
    Let $\alpha\in(0,2)$ and $Y_t$ be a cylindrical $\alpha$-- stable process in $\Rd,$ i.e.,
    \[
       \mE e^{i\langle \xi, Y_t \rangle } = \prod_{k=1}^d e^{-t|\xi_k|^\alpha} = e^{-t\sum_{k=1}^d|\xi_k|^\alpha},\quad \xi\in\Rd, t>0.
    \]
    Let $\theta\in \mathbb{S}^{d-1} = \{y\in\Rd:\: |y|=1\} $ and let $m$ denote a number of non--zero coordinates of $\theta.$
    If $m>\frac{d-\alpha}{\alpha+1}$
    then for all $x\in\Rd, \lambda>0$ we have
    \begin{equation}\label{eq:asymptforcylind}
		p_\lambda(t,x,r\theta)\sim_{t,r}\, \lambda^{\frac{d-m}{\alpha}-m} C_{\alpha,d}\Big(\prod\limits_{\substack{i=1,...,d\\\theta_i\neq0}}|\theta_i|^{-(1+\alpha)}\Big) r^{-m(1+\alpha)},
    \end{equation}
    where \footnote{$A(t,r)\sim_{t,r} B(t,r)$ means, that $\lim\limits_{\substack{r\to\infty\\t\to\infty}}\frac{A(t,r)}{B(t,r)}=1,$ and $A(t,r)\sim_r B(t,r)$ means, that $\lim\limits_{r\to\infty}\frac{A(t,r)}{B(t,r)}=1.$}
    \[
       C_{\alpha,d} = \mathcal{A}_{1,\alpha}^m \Gamma\left(\frac{m-d}{\alpha}+m+1\right)\left(\frac{\Gamma(\frac{1}{\alpha})}{\alpha\pi}\right)^{d-m}.
    \]
\end{theorem}
\begin{proof}

	%Let $Y_t=\left(Y_t^{(1)}, Y_t^{(2)}, \dots, Y_t^{(d)}\right)$, where $\left\{Y_t^{(i)}\right\}$ are independent symmetric stable processes in $\mathbb{R}$ with density function $\tilde{p}(t,x)$. 
    The density of $Y_t$ can be written as
	\[
		   p(t,r\theta)=\prod_{i=1}^{d}\tilde{p}(t,r\theta_i), \quad \theta=\left(\theta_1,\theta_2,\dots,\theta_d\right)\in\mathbb{S}^{d-1}, r>0,
	\]
	where $\tilde{p}$ is the density of the one--dimensional $\alpha$--stable symmetric L\'evy process. Using the scaling property and \eqref{eq:BlumenthalGetoor} for $d=1$ we get  
	\[
		\tilde{p}(t,r\theta_i)\sim_r\left\{
          \begin{array}{ll}
			t^{-1/\alpha}\tilde{p}(1,0) & \text{for } \,\theta_i=0, \\
			\dfrac{\mathcal{A}_{1,\alpha} t}{\left(r|\theta_i|\right)^{1+\alpha}} & \text{for }\, \theta_i\neq0.
		\end{array}\right.
        \]
    Note that we have $\mathcal{A}_{1,\alpha}=\Gamma(\alpha+1)\sin(\pi\alpha/2)/\pi$. It follows that
 	\[
 		\lim\limits_{r\to\infty}\frac{p(t,r\theta)}{r^{-m(1+\alpha)}}=t^{-(d-m)/\alpha}\tilde{p}^{d-m}(1,0)\mathcal{A}_{1,\alpha}^mt^{m}\prod\limits_{\substack{i=1,...,d\\\theta_i\neq0}}|\theta_i|^{-(1+\alpha)}.
 	\]
	%where $m$ is number of non-zero coordinates of $\theta$. 
    Furthermore,
     \begin{align*}
      \frac{p(t,r\theta-x)}{r^{-m(1+\alpha)}} &
      = 
       \frac{\prod_{i=1}^d\tilde{p}(t,r\theta_i - x_i)}{r^{-m(1+\alpha)}} \\
      & =  
      t^{-d/\alpha} \prod\limits_{\substack{i=1,...,d\\ \theta_i\neq 0}} \left(r^{1+\alpha}\tilde{p}\left(1,\frac{r\theta_i-x_i}{t^{1/\alpha}}\right)\right)
      \prod\limits_{\substack{i=1,...,d\\ \theta_i =0}} \tilde{p}\left(1,\frac{x_i}{t^{1/\alpha}}\right) \\
      & \leq 
      c_1 t^{-d/\alpha} \prod\limits_{\substack{i=1,...,d\\ \theta_i\neq 0}} \left(\frac{t^{1/\alpha}r}{|r\theta_i-x_i|}\right)^{1+\alpha}
      \leq c_2 t^{m(1+1/\alpha)-d/\alpha},
    \end{align*}
    provided $r\geq \max\{2|x_i/\theta_i|:\: i=1,...,d,\, \theta_i\neq 0\}.$
    If $m(1+1/\alpha)-d/\alpha>-1$, then the
    assumptions of Lemma~\ref{lemma:asymptotic} hold for
    \[    
		f(r)=r^{-m(1+\alpha)}, \quad
		A(t)=\mathcal{A}_{1,\alpha}^m t^{m(1+1/\alpha)-d/\alpha}\tilde{p}^{d-m}(1,0)\prod\limits_{\substack{i=1,...,d\\\theta_i\neq0}}|\theta_i|^{-(1+\alpha)},
    \]
    and $\tilde{A}(t)=c_3 A(t)$    for some constant $c_3>0.$
	Therefore, we obtain \eqref{eq:asymptforcylind} using the Lemma \ref{lemma:asymptotic} and observing that $\tilde{p}(1,0) = (2\pi)^{-1} \int_{\R} e^{-|\xi|^\alpha}\, \dx[\xi] = (\pi\alpha)^{-1} \Gamma(\alpha^{-1}).$ 
\end{proof}

Note that in the case $m \leq \frac{d-\alpha}{1+\alpha}$ for fixed $r>0$ and sufficiently small $s>0,$ we have
\begin{align*}
   p(s,r\theta) & = (\tilde{p}(s,0))^{d-m} \prod\limits_{\substack{k=1,...,d\\ \theta_k\neq 0}} \tilde{p}(s,r\theta_k)
    = c_1 s^{-\frac{d-m}{\alpha}} \prod\limits_{\substack{k=1,...,d\\ \theta_k\neq 0}} s^{-\frac{1}{\alpha}} \tilde{p}(1,r\theta_k/s^{1/\alpha}) \\
    & \geq c_2 s^{-\frac{d-m}{\alpha}} \prod\limits_{\substack{k=1,...,d\\ \theta_k\neq 0}} (|r\theta_k|^{-1-
    \alpha}s) = c_2 r^{-m(1+\alpha)} \Big(\prod\limits_{\substack{k=1,...,d\\ \theta_k\neq 0}} |\theta_k|^{-1-\alpha}\Big)s^{m-\frac{d-m}{\alpha}},
\end{align*}
hence $p_\lambda(t,x,r\theta) \geq \int_0^t \lambda e^{-\lambda s} p(s,r\theta)\, \dx[s] = \infty,$ for all $t>0, x\in\Rd.$

%%%%%%%%%%%%%%%%%%%%%%%%%%%%%%%%%%%%%%%%%%%%%%%%%%%%%%%%%%%%%%%%%%%%%%%%%%%%%%%%%%%%%%%%%%%%%%%%%%%%%%%%%%%%%%%%%%%%%%%%%%%%%%%%%%%%%%%%%%% 3

\subsection{Stable subordinators.}
    \begin{theorem}
        Let $\beta\in(0,1)$ and $Y_t$ be a $\beta$ -- stable subordinator, i.e., real-valued L\'evy process with non-decreasing sample paths, such that 
	\[
          \mE e^{-s Y_t} = %\exp\left\{\frac{-\beta t}{\Gamma(1-\beta)}\int_0^\infty(1-e^{-s r})r^{-1-\beta}\,\dx[r]\right\} = 
          e^{-t s^\beta},\quad s,t>0, 
	\]
        and $\mP(Y_t<0)=0$ for all $t>0.$ Then, for all $\lambda>0, t>0,$ and $x\in\R,$ the density $p_\lambda(t,x,\cdot)$ of
        the corresponding process with resetting exists and satisfies
        \begin{equation}\label{eq:asfsub}
		   p_\lambda(t,x,u)\sim_{t,u}\,\, \frac{\beta}{\lambda\Gamma(1-\beta)}u^{-1-\beta}.
	\end{equation}
    \end{theorem}
    \begin{proof}
	Let $\eta(t,\cdot)$ denote the transition density of $Y_t$ (we change the notation here to be consistent with Section \ref{section:RelStP}). It follows from \cite[Theorem 1.1(i)]{watanabe} that
	\[
		  \eta(1,u)\leq c_1 (1+u)^{-(1+\beta)},\quad u>0,
	\]
	for some constant $c_1>0,$ and by the scaling property of $\eta(t,u)$ we obtain
	\begin{equation}\label{eq:subord_ineq}
		  \eta(t,u) = t^{-1/\beta} \eta(1,ut^{-1/\beta}) \leq c_1\frac{t}{u^{1+\beta}}, \quad u>0, t>0.
	\end{equation}
	Furthermore, using known (see, e.g., \cite[(4.2.4)]{Zolotariev1999}) series expansion 
        \begin{equation}\label{eq:seriesstab}
           \eta(1,u) = \pi^{-1}\sum_{n=1}^\infty \frac{(-1)^{n-1}\Gamma(n\beta+1)}{n!}\sin(n\beta\pi) u^{-n\beta-1},
           \quad u>0,
        \end{equation}
        and the formula $\sin(\pi x) = \frac{\pi}{\Gamma(x)\Gamma(1-x)}, x\in\R\setminus\mathbb{Z},$ we get 
	\[
		  \eta(t,u)\sim_u \frac{t}{u^{1+\beta}}\frac{\beta}{\Gamma(1-\beta)}, \quad t>0.
	\]	
       Furthermore, 
       \[
          \frac{\eta(t,u-x)}{u^{-1-\beta}} \leq c_1 t\left(\frac{u}{|u-x|}\right)^{1+\beta} \leq c_2 t,
        \]
        provided $u>2|x|.$
	Thus, for $f(u)=u^{-(1+\beta)}$ and $A(t)=t\beta/\Gamma(1-\beta), \tilde A(t)=c_1 t$, 
        the density $\eta$ satisfies the assumptions of Lemma~\ref{lemma:asymptotic} and using it we get \eqref{eq:asfsub}.
    \end{proof}
    For further use, we observe that the scaling property and \eqref{eq:seriesstab} also yield
    \begin{equation}\label{eq:limitsubLev}
        \lim_{t\to 0^+} \frac{\eta(t,u)}{t} = \frac{\beta}{\Gamma(1-\beta)}u^{-1-\beta} = : \mathcal{C_\beta}u^{-1-\beta},\quad u>0.
    \end{equation}
    We note that this is a particular case of a much more general property: $\lim_{t\to 0^+} p(t,x)/t = \nu(x),$ $x\neq 0,$ which holds for sufficiently regular densities of L{\'e}vy processes (here $\nu$ is a density of a~corresponding L{\'e}vy measure).

%%%%%%%%%%%%%%%%%%%%%%%%%%%%%%%%%%%%%%%%%%%%%%%%%%%%%%%%%%%%%%%%%%%%%%%%%%%%%%%%%%%%%%%%%%%%%%%%%%%%%%%%%%%%%%%%%%%%%%%%%%%%%%%%%%%%%%%%%%% 4

\subsection{Isotropic unimodal L\'evy process.}
    Let $Y_t$ be a pure-jump isotropic unimodal L\'evy process in $\Rd,$ i.e.,
        \[
           \mE e^{i\langle \xi, Y_t \rangle} = e^{-t\psi(\xi)}, \quad \xi\in\Rd, t\geq 0,
        \]
        where
        \[
           \psi(\xi) = \int_{\Rd} (1-\cos \langle\xi,y\rangle ) \nu(y)\,\dx[y], \quad \xi\in\Rd,
       \]
       \[
          \int_{\Rd} \nu(y)\, \dx[y] = \infty,  \quad \int_{\Rd} (|y|^2\wedge 1) \nu(y)\,\dx[y]<\infty,
       \]
       and the L\'evy density $\nu$ is isotropic and non-increasing in $|y|$.
    \begin{theorem} Let $Y_t$ be a pure-jump isotropic unimodal L\'evy process in $\Rd$. 
         Assume that
    $\psi$ is regularly varying of index $\alpha\in (0,2)$ at zero, i.e.,
	\[
		\lim\limits_{x\to 0^+}\frac{\psi\left(\lambda x\right)}{\psi\left(x\right)}=\lambda^{\alpha},\quad\lambda>0.
	\]
    Then, for all $\lambda, t >0, x\in\Rd,$ the density $p_\lambda(t,\cdot)$ of the corresponding process with resetting exists  and satisfies
    \begin{equation}\label{eq:unimodalcaseasympt}
       p_\lambda(t,x,y) \sim \frac{\mathcal{A}_{d,\alpha}}{\lambda} |y|^{-d} \psi(|y|^{-1}).
    \end{equation}
    \end{theorem}
    \begin{proof}
        It follows from \cite[Corollary 3 and 7]{Bogdan} that 
	\[
		p(t,z)\leq c_1 t|z|^{-d}\psi\left(|z|^{-1}\right), \quad z\neq0, t>0,
	\]
	for some constant $c_1>0.$ Since $p(t,\cdot)$ is radially non-increasing and $\sqrt{\psi}$ is subadditive, we get
       \[
         p(t,y-x) \leq p(t,y/2)\leq c_2 t |y|^{-d} \psi(|y|^{-1}), \quad t>0,\, |y| > 2|x|.
       \]
    Moreover, in \cite[Theorem 4]{Cygan} it has been shown, that %if $\psi$ is regularly varying of index $\alpha\in(0,2)$ at zero
    %then
	\begin{equation*}
		\lim\limits_{\substack{|y|\to+\infty\\t\psi\left(|y|^{-1}\right)\to0}}\frac{p(t,y)}{|y|^{-d}t\psi\left(|y|^{-1}\right)}=\mathcal{A}_{d,\alpha},
	\end{equation*}
	where $\mathcal{A}_{d,\alpha}$ is defined by formula \eqref{eq:lim_stable}. In particular, since $\psi(y)\xrightarrow{y\to0}0$, we have
	\begin{equation*}
		\lim_{|y|\to\infty}\frac{p(t,y)}{|y|^{-d}\psi\left(|y|^{-1}\right)}=\mathcal{A}_{d,\alpha}t, \quad t>0.
	\end{equation*}
	Thus, %since $Cte^{-\lambda t}$ is integrable, 
        we can use Lemma \ref{lemma:asymptotic} once again to obtain \eqref{eq:unimodalcaseasympt}.
    \end{proof}
	 
\section{Relativistic stable processes}\label{section:RelStP}
Let $m>0$. The relativistic stable process $Y_t$ can be obtained through the subordination procedure, that is, if $B_t$ is a Brownian motion in $\Rd$ and $\theta_t$ is an independent subordinator with transition density $\tilde{\eta}(t,u)=e^{mt}\eta(t,u)e^{-m^{2/\alpha}u}$, where $\eta$ is the transition density of the $\alpha/2$-stable subordinator, then $Y_t = B_{\theta_t}$ (see, e.g., \cite{Ryznar2002}). That is, this process is a pure-jump isotropic unimodal L\'evy process, but $\psi(\xi)=(|\xi|^2+m^{2/\alpha})^{\alpha/2}-m$ varies reguraly at zero with index $2$.

This indicates a following representation of the transition density of $Y_t:$
\begin{equation*}\label{eq:relativistic_density}
    p(t,y)=e^{mt}\int_{0}^{\infty}\left(\frac{1}{4\pi u}\right)^{d/2} e^{-\frac{|y|^2}{4u}}e^{-m^{2/\alpha}u} \eta(t,u) \dx[u], \quad t>0, y\in\Rd.
\end{equation*}
Using \eqref{eq:subord_ineq} and \cite[10.32.10]{NIST} we get
    \begin{equation} \label{eq:reldensest}
      \begin{split}
		p(t,y) & \leq c_1 e^{mt}\int_0^\infty\left(\frac{1}{4\pi u}\right)^{d/2} e^{-\frac{|y|^2}{4u}}e^{-m^{2/\alpha}u} \frac{t}{u^{1+\alpha/2}}\,\dx[u] \\
        & = \frac{c_2 t e^{mt}K_{\frac{d+\alpha}{2}}\left(m^{1/\alpha}|y|\right)}{|y|^{\frac{d+\alpha}{2}}},
        \quad t>0, y\in\Rd\setminus\{ 0 \},
      \end{split}
    \end{equation}
    for some positive constants $c_1, c_2,$ where 
    \begin{align*}
       K_\mu(r) & = \frac12 \left(\frac{r}{2}\right)^\mu \int_0^\infty u^{-\mu-1} \exp\left(-u-\frac{r^2}{4u}\right)\, \dx[u] \\
       &= \frac{1}{2^{\mu+1}} \int_0^\infty s^{-\mu-1} \exp\left(-r\left(s+\frac{1}{4s}\right)\right)\,\dx[s], \quad \mu>0, r>0,
    \end{align*}
    is the modified Bessel function of the second kind.
    The density of the L\'evy measure of $Y_t$ is given by (see \cite[Theorem 30.1]{Sato68})
    \begin{equation}\label{eq:relLevy}
         \begin{split}
             \nu(y) & =  \frac{\alpha}{2(4\pi)^{\frac{d}{2}}\Gamma(1-\frac{\alpha}{2})} \int_0^\infty  e^{-\frac{|y|^2}{4u}} e^{-m^{2/\alpha }u} u^{-1-\frac{d+\alpha}{2}}\, \dx[u] \\
             & = \frac{\alpha m^{\frac{d+\alpha}{2\alpha}}}{2^{\frac{d-\alpha}{2}} \pi^{\frac{d}{2}}\Gamma(1-\frac{\alpha}{2})} \frac{K_{\frac{d+\alpha}{2}}(m^{1/\alpha}|y|)}{|y|^{\frac{d+\alpha}{2}}}, \quad y\in\Rd\setminus\{ 0 \}.
         \end{split}
    \end{equation}

\subsection{\texorpdfstring{$m<\lambda$}{mniejsza}}\label{ex:relativistic}

    \begin{proof}[Proof of Theorem \ref{th:relatproc}, Part 1.]
        It follows from \cite[Theorem 4]{KaletaSztonyk2019} that 
        \[
             \lim_{|z|\to\infty} \frac{p(t,z)}{t\nu(z)} = e^{mt}, \quad t>0,
        \]
        hence, for $f(r)=r^{-\frac{d+\alpha}{2}} K_{\frac{d+\alpha}{2}}\left(m^{1/\alpha}r\right)$ and all $t>0,$ by \eqref{eq:relLevy}, we obtain
    \begin{align*}
		\begin{split}
            \lim_{|z|\to\infty}\frac{p(t,z)}{f(|z|)} &=
            \lim_{|z|\to\infty}\frac{p(t,z)}{\nu(z)}\cdot
            \frac{\nu(z)}{K_{\frac{d+\alpha}{2}}(m^{1/\alpha}|z|)|z|^{-\frac{d+\alpha}{2}}}\phantom{=:A(t)}\\ &=
            te^{mt}\cdot
            \frac{\alpha m^{\frac{d+\alpha}{2\alpha}}}{2^{\frac{d-\alpha}{2}}\pi^{\frac{d}{2}}\Gamma(1-\frac{\alpha}{2})} =:
            A(t).
		\end{split}
	\end{align*}%\label{eq:A_relat}
     Furthermore, we have (see \cite[10.25.3]{NIST})
        \[
		   \lim_{r\to\infty}\frac{K_{\mu}\left(r\right)}{\sqrt\frac{\pi}{2r}e^{-r}}=1,
        \] 
     and it follows from \eqref{eq:reldensest} that $p(t,y-x) \leq c_1 f(|y-x|) A(t) \leq c_2 f(|y|) A(t),$ for some constant $c_1>0,$ $c_2=c_2(x),$ and suitable large $y,$ hence using Lemma~\ref{lemma:asymptotic} again we get
	\begin{align*}
        p_\lambda(t,x,y) & \sim
        |y|^{-\frac{d+\alpha}{2}}K_{\frac{d+\alpha}{2}}\left(m^{1/\alpha}|y|\right) 
        \frac{\lambda \alpha m^{\frac{d+\alpha}{2\alpha}}}{2^{\frac{d-\alpha}{2}}\pi^{\frac{d}{2}}\Gamma(1-\frac{\alpha}{2})(\lambda-m)^2} \\
        & \sim \frac{\lambda\alpha m^{\frac{d+\alpha-1}{2\alpha}}}{2^{\frac{d-\alpha+1}{2}}\pi^{\frac{d-1}{2}}\Gamma(1-\frac{\alpha}{2})(\lambda-m)^2} |y|^{-\frac{d+\alpha+1}{2}} e^{-m^{1/\alpha}|y|}.
    \end{align*}
   \end{proof}

 \subsection{\texorpdfstring{$m=\lambda$}{rowny}}\label{ex:relativistic_m=l}

 In this case we use the standard Laplace method of estimating the integrals which we recall here without proof.
\begin{lemma}[Generalized Laplace Method]\label{thm:laplace}
     Let $-\infty<a<b<\infty$, and $f\in C^2((a,b))$ with unique global maximum at $x_0\in(a,b),$ $h(x)\in C((a,b))$ and $h(x)\geq 0$ for $x\in(a,b)$. If $e^{f(x)}h(x)$ is integrable on $(a,b)$, then
     \begin{equation*}
        \int_a^b h(x)e^{rf(x)} \dx \sim_r h(x_0)\sqrt{\frac{2\pi}{r|f^{\prime\prime}(x_0)|}} e^{rf(x_0)}.
    \end{equation*}
\end{lemma}

%Later we will use the saddle point type method for estimations. Let us recall firstly important Theorem.

\begin{lemma}\label{lm:intasym}
    Let $a, b >0$ and $f(s) = \frac{a}{s} + b s,\, s\in (0,\infty).$ Let $\gamma > 0,$ and a positive
    function $g\in C_b(\R_+)$ be such that for some $\kappa\in [0,\infty),$ and $g_0, g_\infty \in (0,\infty)$ we have
    \begin{equation}\label{eq:g_assum}
        \lim_{s\to 0^+} \frac{g(s)}{s^{\kappa}} = g_0, \quad \lim_{s \to \infty} g(s) = g_\infty.
    \end{equation}
    Then
    \begin{equation}\label{eq:CorLapl}
       \int_0^\infty s^{-\gamma}e^{-rf(s)}g\left(\frac{t^{2/\alpha}}{rs}\right)\, \dx[s]
       \sim \sqrt{\frac{\pi}{a}}\left(\frac{a}{b}\right)^{\frac34 - \frac{\gamma}{2}} \frac{e^{-2r\sqrt{ab}}}{\sqrt{r}} g\left(\frac{t^{2/\alpha}\sqrt{b}}{r\sqrt{a}}\right).
    \end{equation}
    Furthermore, the above asymptotic equality holds also for
    any positive function $g\in C_b(\R_+)$ such that $\lim_{s\to 0^+} \frac{g(s)}{s^\kappa} = g_0$ 
    (without assuming the existence of a positive limit in infinity) provided $\frac{t}{r^{\alpha/2}} <M$ for some constant $M>0.$
    
\end{lemma}
\begin{proof}
    The function $f$ is continuous on $(0,\infty)$ with global
    minimum at $s_0 = \sqrt{\frac{a}{b}}.$ We observe also that
    \begin{equation}\label{eq:fbeh}
        f(s_0-h) > f(s_0+h)  > f(s_0-h/2),
    \end{equation}
    for all $h\in (0,s_0).$
    Thus, for every $\varepsilon \in (0,s_0),$ we have
    \begin{align}\label{eq:IntEst1}
        \begin{split}
           \int_{|s-s_0|<\varepsilon} s^{-\gamma}e^{-rf(s)} g\left(\frac{t^{2/\alpha}}{rs}\right)\, \dx[s] & \geq 
           \int_{|s-s_0|<\varepsilon/2} (s_0+\varepsilon/2)^{-\gamma}e^{-rf(s_0-\varepsilon/2)} \, \dx[s] 
           \inf_{|s-s_0|<\varepsilon/2} g\left(\frac{t^{2/\alpha}}{rs}\right) \\
           & = \varepsilon (s_0+\varepsilon/2)^{-\gamma}e^{-rf(s_0-\varepsilon/2)} \inf_{|s-s_0|<\varepsilon/2} g\left(\frac{t^{2/\alpha}}{rs}\right).
        \end{split}
    \end{align}
    We note that from the integrability of $s^{-\gamma}e^{-f(s)}$ it follows that there exists $c_1>s_0$ such that
    \[
       \int_{s_0+\varepsilon}^{c_1} s^{-\gamma}e^{-f(s)} \, \dx[s]
       \geq \int_{c_1}^\infty  s^{-\gamma}e^{-f(s)} \, \dx[s].
    \]
    Then, for every $r>1$ we get
    \begin{align*}
       \int_{s_0+\varepsilon}^{c_1} s^{-\gamma}e^{-rf(s)} \, \dx[s] & \geq e^{-(r-1)f(c_1)} \int_{s_0+\varepsilon}^{c_1} s^{-\gamma}e^{-f(s)} \, \dx[s] 
       \geq e^{-(r-1)f(c_1)} \int_{c_1}^\infty s^{-\gamma}e^{-f(s)} \, \dx[s] \\
       & \geq \int_{c_1}^\infty s^{-\gamma}e^{-rf(s)} \, \dx[s],
    \end{align*}
    hence
    \begin{align}\label{eq:IntEst2}
      \begin{split}
        \int_{s_0+\varepsilon}^{\infty} s^{-\gamma}e^{-rf(s)} g\left(\frac{t^{2/\alpha}}{rs}\right)\, \dx[s] &
        \leq 2\|g\|_\infty \int_{s_0+\varepsilon}^{c_1} s^{-\gamma}e^{-rf(s)} \dx[s] \\
        & \leq 2c_1 \|g\|_\infty (s_0+\varepsilon)^{-\gamma}e^{-rf(s_0+\varepsilon)},
      \end{split}
    \end{align}
    where, as usual, $\|g\|_\infty = \sup_{s\in (0,\infty)} g(s).$ Furthermore, the function $s\to s^{-\gamma}e^{-rf(s)}$ is increasing on $\left(0,\frac{\sqrt{\gamma^2+4r^2ab}-\gamma}{2rb}\right],$ hence we have
    \begin{align}\label{eq:IntEst3}
      \begin{split}
        \int_{0}^{s_0-\varepsilon} s^{-\gamma}e^{-rf(s)} g\left(\frac{t^{2/\alpha}}{rs}\right)\, \dx[s]
        & \leq \| g\|_\infty (s_0-\varepsilon) \tilde{s}^{-\gamma} e^{-rf(\tilde{s})},
      \end{split}
    \end{align}
    where $\tilde{s}=\min\left\{s_0-\varepsilon,\frac{\sqrt{\gamma^2+4r^2ab}-\gamma}{2rb} \right\}.$
    There exists $r_0=r_0(a,b,\varepsilon,\gamma)$ such that $\tilde{s}=s_0-\varepsilon$ for all $r>r_0,$ and from \eqref{eq:IntEst1}, \eqref{eq:IntEst2}, \eqref{eq:IntEst3} and \eqref{eq:fbeh} we get
    \begin{align*}
        \frac{\int_{|s-s_0|>\varepsilon} s^{-\gamma}e^{-rf(s)} g\left(\frac{t^{2/\alpha}}{rs}\right)\, \dx[s]}{\int_{{|s-s_0| <\varepsilon}} s^{-\gamma}e^{-rf(s)} g\left(\frac{t^{2/\alpha}}{rs}\right)\, \dx[s]}
        & \leq \frac{c_2 \|g\|_\infty e^{-rf(s_0+\varepsilon)}}{ e^{-rf(s_0-\varepsilon/2)}\inf\limits_{|s-s_0|<\varepsilon/2} g\left(\frac{t^{2/\alpha}}{rs}\right)} \\
        & = \frac{c_2 \|g\|_\infty e^{-r(f(s_0+\varepsilon)-f(s_0-\varepsilon/2))}}{\inf\limits_{|s-s_0|<\varepsilon/2} g\left(\frac{t^{2/\alpha}}{rs}\right)},
    \end{align*}
    for all $r>\max\{1,r_0\},$ with some constant $c_2=c_2(a,b,\varepsilon,\gamma).$
    It follows from \eqref{eq:g_assum} that there exists $\delta>0$ such that $g(s) \geq \frac{g_0}{2} s^\kappa$ for $s\in (0,\delta),$ hence, if $\frac{t^{2\alpha}}{r} \leq \delta (s_0-\varepsilon/2)$ then $\inf\limits_{|s-s_0|<\varepsilon/2} g\left(\frac{t^{2/\alpha}}{rs}\right) \geq \frac12 g_0 (s_0+\varepsilon/2)^{-\kappa} (t^{2/\alpha}/r)^\kappa.$ If $\frac{t^{2\alpha}}{r} > \delta (s_0-\varepsilon/2)$ then $\inf\limits_{|s-s_0|<\varepsilon/2} g\left(\frac{t^{2/\alpha}}{rs}\right) \geq \inf\limits_{u > \frac{\delta(s_0-\varepsilon/2)}{s_0+\varepsilon/2}} g\left(u\right) > 0$.
    This and \eqref{eq:fbeh} yields
    \[
      \lim\limits_{\substack{r\to\infty\\t\to\infty}} \frac{\int_{|s-s_0|>\varepsilon} s^{-\gamma}e^{-rf(s)} g\left(\frac{t^{2/\alpha}}{rs}\right)\, \dx[s]}{\int_{{|s-s_0| <\varepsilon}} s^{-\gamma}e^{-rf(s)} g\left(\frac{t^{2/\alpha}}{rs}\right)\, \dx[s]} = 0,
    \]
    and
    \[
      \lim\limits_{\substack{r\to\infty\\t\to\infty}} \frac{\int_0^\infty s^{-\gamma}e^{-rf(s)} g\left(\frac{t^{2/\alpha}}{rs}\right)\, \dx[s]}{\int_{{|s-s_0| <\varepsilon}} s^{-\gamma}e^{-rf(s)} g\left(\frac{t^{2/\alpha}}{rs}\right)\, \dx[s]} = 1,
    \]
    for every $\varepsilon \in (0,s_0).$

    It follows from the Generalized Laplace Method (Lemma \ref{thm:laplace}) that
    \[
      \int_{{|s-s_0| <\varepsilon}} s^{-\gamma}e^{-rf(s)} \, \dx[s]
      \sim_r
      s_0^{-\gamma} \sqrt{\frac{2\pi}{r|f''(s_0)|}}e^{-rf(s_0)},
    \]
    hence
    \begin{equation}\label{eq:limsupcomp}
      \limsup\limits_{\substack{r\to\infty\\t\to\infty}}
      \frac{\int_0^\infty s^{-\gamma}e^{-rf(s)} g\left(\frac{t^{2/\alpha}}{rs}\right)\, \dx[s]}{g\left(\frac{t^{2/\alpha}}{rs_0}\right) s_0^{-\gamma} \sqrt{\frac{2\pi}{r|f''(s_0)|}}e^{-rf(s_0)}}
      \leq \limsup\limits_{\substack{r\to\infty\\t\to\infty}} \frac{\sup\limits_{|s-s_0|<\varepsilon} g\left(\frac{t^{2/\alpha}}{rs}\right)}{g\left(\frac{t^{2/\alpha}}{rs_0}\right)}
    \end{equation}
    and similarly,
    \begin{align*}
      \liminf\limits_{\substack{r\to\infty\\t\to\infty}}
      \frac{\int_0^\infty s^{-\gamma}e^{-rf(s)} g\left(\frac{t^{2/\alpha}}{rs}\right)\, \dx[s]}{g\left(\frac{t^{2/\alpha}}{rs_0}\right) s_0^{-\gamma} \sqrt{\frac{2\pi}{r|f''(s_0)|}}e^{-rf(s_0)}}
      & \geq \liminf\limits_{\substack{r\to\infty\\t\to\infty}} \frac{\inf\limits_{|s-s_0|<\varepsilon}g\left(\frac{t^{2/\alpha}}{rs}\right)}{g\left(\frac{t^{2/\alpha}}{rs_0}\right)}.
    \end{align*}
    For every $\delta>0,$ using \eqref{eq:g_assum}, we choose $\eta>0$ such that $1-\delta < \frac{g(u)}{g_0 u^\kappa}<1+\delta$ for $u\in (0,\eta).$ We get
    \[
      \frac{\sup\limits_{|s-s_0|<\varepsilon} g\left(\frac{t^{2/\alpha}}{rs}\right)}{g\left(\frac{t^{2/\alpha}}{rs_0}\right)} \leq \left(\frac{s_0}{s_0-\varepsilon}\right)^\kappa \frac{1+\delta}{1-\delta},
    \]
    provided $\frac{t^{2/\alpha}}{r} \leq \frac12 s_0\eta$ and
    $\varepsilon \in (0,\frac12 s_0).$

    Similarly, since $\lim_{s\to\infty} g(s) = g_\infty,$
    we can choose $C>0$ such that 
    \begin{equation}\label{eq:gests}
       \frac{\sup\limits_{|s-s_0|<\varepsilon} g\left(\frac{t^{2/\alpha}}{rs}\right)}{g\left(\frac{t^{2/\alpha}}{rs_0}\right)} \leq 1+\delta,
    \end{equation}
    provided $\frac{t^{2/\alpha}}{r} \geq s_0 C.$
 
    The function $g$ is uniformly continuous on $[\frac13\eta,2C]$ for all $C>\eta >0.$ Therefore, we can choose $\varepsilon_0 = \varepsilon_0(g,s_0,\delta)>0$ such that
    \eqref{eq:gests} holds also for $\frac12 s_0\eta < \frac{t^{2/\alpha}}{r} <s_0C,$ provided $\varepsilon < \varepsilon_0.$
    
    We obtain
    \[
       \limsup\limits_{\substack{r\to\infty\\t\to\infty}}\frac{\sup\limits_{|s-s_0|<\varepsilon} g\left(\frac{t^{2/\alpha}}{rs}\right)}{g\left(\frac{t^{2/\alpha}}{rs_0}\right)} \leq \left(\frac{s_0}{s_0-\varepsilon}\right)^{\kappa}\frac{1+\delta}{1-\delta}, \quad \varepsilon<\varepsilon_0,
    \]
    and since $\delta>0$ is arbitrary, by \eqref{eq:limsupcomp} we have
    \[
       \limsup\limits_{\substack{r\to\infty\\t\to\infty}}
      \frac{\int_0^\infty s^{-\gamma}e^{-rf(s)} g\left(\frac{t^{2/\alpha}}{rs}\right)\, \dx[s]}{g\left(\frac{t^{2/\alpha}}{rs_0}\right) s_0^{-\gamma} \sqrt{\frac{2\pi}{r|f''(s_0)|}}e^{-rf(s_0)}} \leq 1.
    \]
    Similarly, we obtain
    \[
       \liminf\limits_{\substack{r\to\infty\\t\to\infty}}
      \frac{\int_0^\infty s^{-\gamma}e^{-rf(s)} g\left(\frac{t^{2/\alpha}}{rs}\right)\, \dx[s]}{g\left(\frac{t^{2/\alpha}}{rs_0}\right) s_0^{-\gamma} \sqrt{\frac{2\pi}{r|f''(s_0)|}}e^{-rf(s_0)}} \geq 1,
    \]
    and \eqref{eq:CorLapl} follows by computing
    $f(s_0) = 2\sqrt{ab},$ and $f''(s_0) = 2a\left(\frac{a}{b}\right)^{-3/2}.$

    Now, assume $\frac{t}{r^{\alpha/2}}<M$ for some $M>0$
    and $g\in C_b(\R_+),$ $g>0$ is such that $\lim_{s\to 0^+} \frac{g(s)}{s^\kappa} = g_0.$ For $\delta>0$ we define
    \[
      \tilde{g}(s) = \left\{
        \begin{array}{cc}
           g(s)  &  \text{if }\,\, s\leq M^{2/\alpha}/\delta,\\
           g(M^{2/\alpha}/\delta)  & \text{if }\,\, s\geq M^{2/\alpha}/\delta. 
        \end{array}
      \right.
    \] 
    Of course, we have
    \[
       \lim_{s\to 0^+} \frac{\tilde{g}(s)}{s^{\kappa}} = g_0 > 0, \quad \lim_{s\to \infty} \tilde{g}(s) = g(M^{2/\alpha}/\delta) > 0,
    \]
    so we can use \eqref{eq:CorLapl} for 
    $\tilde{g}.$
    In particular, for some constant $c_1=c_1(a,b)>0,$ we have
    \[
       \int_0^\infty s^{-\gamma} e^{-rf(s)} \tilde{g}\left(\frac{t^{2/\alpha}}{rs}\right)\dx[s] \geq c_1 
       \frac{e^{-2r\sqrt{ab}}}{\sqrt{r}} \tilde{g}\left(\frac{t^{2/\alpha}\sqrt{b}}{r\sqrt{a}}\right),
    \]
    for sufficiently large $r$ and $t.$
    
    We recall that $s\to s^{-\gamma}e^{-rf(s)}$ is
    increasing on $(0,s_1],$ where \\
 $s_1 = \frac{\sqrt{\gamma^2+4r^2ab}-\gamma}{2rb},$ hence if $\delta< \frac12 \sqrt{\frac{a}{b}}=\frac12 s_0$ and $r$ is such that $s_1 > \frac12 \sqrt{\frac{a}{b}}=\frac12 s_0,$ then we have
    \begin{align*}
       &\frac{\int_0^\infty s^{-\gamma} e^{-rf(s)} \left| g\left(\frac{t^{2/\alpha}}{rs}\right) - \tilde{g}\left(\frac{t^{2/\alpha}}{rs}\right)\right|\dx[s]}{\int_0^\infty s^{-\gamma} e^{-rf(s)} \tilde{g}\left(\frac{t^{2/\alpha}}{rs}\right)\dx[s]} 
        \leq  \frac{\int_0^{\delta} s^{-\gamma} e^{-rf(s)} \left| g\left(\frac{t^{2/\alpha}}{rs}\right) - \tilde{g}\left(\frac{t^{2/\alpha}}{rs}\right)\right|\dx[s]}{c_1 
       e^{-2r\sqrt{ab}}r^{-1/2} \tilde{g}\left(\frac{t^{2/\alpha}\sqrt{b}}{r\sqrt{a}}\right)} \\
       &\leq  \frac{\delta^{-\gamma} e^{-rf(\delta)} \int_0^{\delta} \left| g\left(\frac{t^{2/\alpha}}{rs}\right) - \tilde{g}\left(\frac{t^{2/\alpha}}{rs}\right)\right|\dx[s]}{c_1 
       e^{-2r\sqrt{ab}}r^{-1/2} \tilde{g}\left(\frac{t^{2/\alpha}\sqrt{b}}{r\sqrt{a}}\right)} 
       \leq  \frac{2\| g\|_\infty \delta^{-\gamma+1} e^{-rf(\delta)}}{c_1 
       e^{-2r\sqrt{ab}}r^{-1/2} \tilde{g}\left(\frac{t^{2/\alpha}\sqrt{b}}{r\sqrt{a}}\right)}.
    \end{align*}
    This quotient tends to $0$ as
    $t,r \to \infty$ because $\tilde{g}\left(\frac{t^{2/\alpha}\sqrt{b}}{r\sqrt{a}}\right) \geq c_2\left( 1 \wedge \frac{t^{2/\alpha}}{r}\right)^\kappa$ for all $t,r>0$ and $f(\delta)>2\sqrt{ab}$ for all $\delta\neq \sqrt{\frac{a}{b}}.$

    This yields
     \[
       \lim\limits_{\substack{r\to\infty, t\to\infty \\ t/r^{\alpha/2}<M}} \frac{\int_0^\infty s^{-\gamma} e^{-rf(s)} g\left(\frac{t^{2/\alpha}}{rs}\right)\, \dx[s]}{\int_0^\infty s^{-\gamma} e^{-rf(s)} \tilde{g}\left(\frac{t^{2/\alpha}}{rs}\right)\,\dx[s]} = 1,\\
    \]
    hence,
    \[
      \int_0^\infty s^{-\gamma} e^{-rf(s)} g\left(\frac{t^{2/\alpha}}{rs}\right) \dx[s]
      \sim 
      \sqrt{\frac{\pi}{a}} \left(\frac{a}{b}\right)^{\frac34 - \frac{\gamma}{2}}\frac{e^{-2r\sqrt{ab}}}{\sqrt{r}} \tilde{g}\left(\frac{t^{2/\alpha}\sqrt{b}}{r\sqrt{a}}\right)
    \]
    and the Lemma follows by the observation
    that $\tilde{g}\left(\frac{t^{2/\alpha}\sqrt{b}}{r\sqrt{a}}\right) = g\left(\frac{t^{2/\alpha}\sqrt{b}}{r\sqrt{a}}\right)$
    for $\delta<\frac12 \sqrt{\frac{a}{b}}$ and $\frac{t}{r^{\alpha/2}}\leq M.$
\end{proof}

In the following Corollary we obtain the asymptotic behavior of
$p$ which is crucial for investigations of $p_\lambda.$ 
%The presence of the factor $e^{mt}$ makes it much more difficult in the case $m\geq \lambda.$

\begin{corollary}\label{reldensprop}
    
    If $\frac{t}{|y|^{\alpha/2}}<M$ for some $M>0$, then
    \begin{equation}\label{eq:reldensasym}
        p(t,y)\sim \left(\frac{m^{1/\alpha}}{2\pi}\right)^{\frac{d-1}{2}} e^{mt-m^{1/\alpha}|y|} |y|^{-\frac{d+1}{2}}\eta\left(\dfrac{t}{|y|^{\alpha/2}}\sqrt{2^\alpha m},1\right).
    \end{equation}
\end{corollary}
\begin{proof}
    We note that
        \begin{equation}\label{eq:p_exp}
            \begin{split}
                p(t,y) &= 
                e^{mt}\int_0^\infty \left(4\pi u\right)^{-d/2} e^{\frac{-|y|^2}{4u}-m^{2/\alpha}u} \eta(t,u)\dx[u]\\ &=
                e^{m t}\left(4\pi \right)^{-d/2} |y|^{-d/2} \int_0^\infty s^{-d/2} e^{-|y|\left(\frac{1}{4s}+m^{2/\alpha}s\right)} \eta\left(\frac{t}{|y|^{\alpha/2}},s\right)\dx[s] \\&=
                e^{m t}\left(4\pi \right)^{-d/2} |y|^{-d/2} \int_0^\infty s^{-d/2-1} e^{-|y|\left(\frac{1}{4s}+m^{2/\alpha}s\right)} \eta\left(\frac{t}{|y|^{\alpha/2}s^{\alpha/2}},1\right)\dx[s].
            \end{split}
        \end{equation}
    Let $g(s) = \eta(s^{\alpha/2},1).$ We have $\lim_{s\to 0^+} \frac{g(s)}{s^{\alpha/2}} = \mathcal{C}_{\alpha/2} > 0$
    (see \eqref{eq:limitsubLev}).
    Using Lemma \ref{lm:intasym} for $a=\frac14,$ $b=m^{2/\alpha},$ $\gamma=1+\frac{d}{2}$ and $g$ we get \eqref{eq:reldensasym}.
    \end{proof}
    \begin{proof}[Proof of Theorem \ref{th:relatproc}, Part 2.] 
    We assume here that $m=\lambda$. We have
        \begin{align*}%\label{eq:tilde_g}
            \begin{split}
            \int_{0}^{t}m e^{-m s}p(s,y)\dx[s] 
            & =
            m\left(4\pi \right)^{-d/2} |y|^{-d/2} \int_0^t \int_0^\infty u^{-d/2-1} e^{-|y|\left(\frac{1}{4u}+um^{2/\alpha}\right)} \eta\left(\frac{s}{|y|^{\alpha/2}u^{\alpha/2}},1\right)\,\dx[u]\dx[s] \\
            & =
             m\left(4\pi \right)^{-d/2} |y|^{-\frac{d-\alpha}{2}} \int_0^\infty u^{-\frac{d-\alpha+2}{2}} e^{-|y|\left(\frac{1}{4u}+um^{2/\alpha}\right)} g\left(\frac{t^{2/\alpha}}{|y|u}\right)\,\dx[u],
        \end{split}   
    \end{align*}
    where
    \[
      g(s) = \int_0^{s^{\alpha/2}} \eta(z,1)\, \dx[z],\quad s>0.
    \]
    We have
    \[
      \lim_{s\to 0^+} \frac{g(s)}{s^{\alpha}} = \frac12 \mathcal{C}_{\alpha/2}  >0,
      \quad \lim_{s\to\infty} g(s) = \int_0^\infty \eta(z,1)\, \dx[z] < \infty, 
    \]
    and using Lemma \ref{lm:intasym} we get
    \begin{equation}\label{eq:intasymptot}
      \int_0^t me^{-ms}p(s,y)\, \dx[s] \sim \left(\frac{m^{1/\alpha}}{2}\right)^{\frac{d+\alpha-1}{2}}
      \pi^{-\frac{d-1}{2}}|y|^{-\frac{d-\alpha+1}{2}}e^{-m^{1/\alpha}|y|}g\left(\frac{2t^{2/\alpha}m^{1/\alpha}}{|y|}\right).
    \end{equation}
    From Corollary \ref{reldensprop} we have
    \[
      e^{-mt}p(t,y-x) \sim \left(\frac{m^{1/\alpha}}{2\pi}\right)^{\frac{d-1}{2}} |y-x|^{-\frac{d+1}{2}} e^{-m^{1/\alpha}|y-x|} \eta\left(\dfrac{t}{|y-x|^{\alpha/2}}\sqrt{2^\alpha m},1\right),
    \]
    provided $\frac{t}{|y-x|^{\alpha/2}}\leq 2M,$
    for some $M>0.$ Hence,
    \[            
        \lim\limits_{\substack{|y|\to\infty,t\to\infty\\ t/|y|^{\alpha/2}\leq M}}
        \frac{e^{-m t}p(t,y-x)}{\displaystyle\int_0^t m e^{-m s}p(s,y)\,\dx[s]} 
        = \lim\limits_{\substack{|y|\to\infty,t\to\infty\\ t/|y|^{\alpha/2}\leq M}}
        \frac{c_1\, \eta\left(\frac{t}{|y-x|^{\alpha/2}}\sqrt{2^\alpha m},1\right)}{ \displaystyle\int_0^{\frac{t}{|y|^{\alpha/2}}\sqrt{2^\alpha m}}\eta(s,1)\, \dx[s]} \frac{e^{-m^{1/\alpha}(|y-x|-|y|)}|y-x|^{-\frac{d+1}{2}}}{|y|^{-\frac{d+1-\alpha}{2}}} = 0,
    \]
    because there exists $\delta>0$ such that $\frac12 \mathcal{C}_{\alpha/2} \leq \frac{\eta(u,1)}{u}\leq \frac32 \mathcal{C}_{\alpha/2}$ for $u\in (0,\delta)$ and the both functions: $\eta(u,1), \int_0^u \eta(s,1)\, \dx[s]$ are bounded below and above by some positive constants on $[\delta/2, 2M\sqrt{2^\alpha m}]$. 

    In the case $\frac{t}{|y|^{\alpha/2}}\geq M,$ using \eqref{eq:subord_ineq}, for
    $|y|>|x|,$ we get
    \[
      \eta\left(\frac{t}{|y-x|^{\alpha/2}s^{\alpha/2}},1\right)
      \leq c_2 \left(\frac{t}{|y-x|^{\alpha/2}s^{\alpha/2}}\right)^{-2/\alpha} \leq
      2c_2 M^{-2/\alpha} s, \quad s>0,
    \]
    and from \eqref{eq:p_exp} and Lemma \ref{lm:intasym} it follows that
    \[
      p(t,y-x)  \leq c_3 e^{mt} |y-x|^{-d/2} 
      \int_0^\infty s^{-d/2} e^{-|y-x|\left(\frac{1}{4s}+m^{2/\alpha}s\right)}\, \dx[s] 
      \leq c_4 e^{mt} |y-x|^{-\frac{d+1}{2}}e^{-m^{1/\alpha}|y-x|},
    \]
    for sufficiently large $t$ and $|y|$. We obviously have
    \[
      \int_0^{\frac{t}{|y|^{\alpha/2}}\sqrt{2^\alpha m}} \eta(s,1)\, \dx[s] \geq \int_0^{M\sqrt{2^\alpha m}}\eta(s,1)\,\dx[s] > 0,
    \]
    and this and \eqref{eq:intasymptot} yield
    \[
        \lim\limits_{\substack{|y|\to\infty,t\to\infty\\ t/|y|^{\alpha/2} > M}}
        \frac{e^{-m t}p(t,y-x)}{\displaystyle\int_0^t m e^{-m s}p(s,y)\,\dx[s]} = 0.
    \]

    Finally, in both cases, $e^{-m t} p(t,y-x)  = o\left(\int_0^t m e^{-m s} p(s,y)\dx[s]\right)$, hence
    \eqref{eq:intasymptot} yields
    \[
        p_m(t,x,y) \sim 
        \left(\frac{m^{1/\alpha}}{2}\right)^{\frac{d+\alpha-1}{2}}
      \pi^{-\frac{d-1}{2}}|y|^{-\frac{d-\alpha+1}{2}}e^{-m^{1/\alpha}|y|}
         \int_0^{\frac{t}{|y|^{\alpha/2}}\sqrt{2^\alpha m}}\eta(z,1)\, \dx[z].
    \]
    
    \end{proof}

\subsection{\texorpdfstring{$m>\lambda$}{odwrotna}}
  Let $\lambda < m.$ In this case we consider only 
       $t$ and $y$ such that $\frac{t}{|y|^{\alpha/2}} < M$
       for some constant $M>0$. In what follows we denote
       $\kappa(t,y):=\frac{t}{|y|^{\alpha/2}}\sqrt{2^\alpha m}.$

We will also often use the fact that by \eqref{eq:limitsubLev} and the continuity of $\eta(\cdot,1)$ for every $T>0$ there exists a constant
$c_0=c_0(\alpha,T)$ such that
  \begin{equation}\label{eq:basicestofeta}
      c_0^{-1} s \leq \eta(s,1) \leq c_0 s, \quad s\in (0,T).
  \end{equation}

  \begin{lemma}\label{lemma:liminfpositive}
        For all $b>a\geq0$ and $c>0$ we have
        \begin{equation}\label{eq:liminfpositive}
        \liminf\limits_{\substack{t\to\infty,\,|y|\to\infty\\ t/|y|^{\alpha/2}<M}}
             \frac{\displaystyle\int_{at}^{bt}  e^{-m s} p(s,y) \dx[s]}{\displaystyle \int_0^{ct} e^{-m s} p(s,y) \dx[s]} =
            \liminf\limits_{\substack{t\to\infty,\,|y|\to\infty\\ t/|y|^{\alpha/2}<M}}\ 
            \frac{\displaystyle \int_{a\kappa(t,y)}^{b\kappa(t,y)}\eta(z,1) \dx[z]}
            {\displaystyle \int_{0}^{c\kappa(t,y)}  \eta(z,1) \dx[z]}  > 0.
        \end{equation}
    \end{lemma}
    \begin{proof} Using the representation \eqref{eq:p_exp} we get
    \begin{align*}
        \int_{at}^{bt} e^{-ms}p(s,y)\, \dx[s] & =
        \left(4\pi \right)^{-d/2} |y|^{-d/2} \int_0^\infty u^{-d/2-1} e^{-|y|\left(\frac{1}{4u}+um^{2/\alpha}\right)} \int_{at}^{bt} \eta\left(\frac{s}{|y|^{\alpha/2}u^{\alpha/2}},1\right)\dx[s]\dx[u] \\
        & = \left(4\pi \right)^{-d/2} |y|^{-\frac{d-\alpha}{2}} \int_0^\infty u^{-\frac{d-\alpha+2}{2}} e^{-|y|\left(\frac{1}{4u}+um^{2/\alpha}\right)} \int_{\frac{at}{|y|^{\alpha/2 }u^{\alpha/2}}}^{\frac{bt}{|y|^{\alpha/2 }u^{\alpha/2}}} \eta(z,1)\, \dx[z]\dx[u],
    \end{align*}
    and from Lemma \ref{lm:intasym} for $g(s) = \int_{as^{\alpha/2}}^{bs^{\alpha/2}} \eta(z,1)\, \dx[z]$ we obtain
    \[
        \int_{at}^{bt} e^{-ms} p(s,y)\,\dx[s] \sim
        \frac{1}{m}\left(\frac{m^{1/\alpha}}{2}\right)^{\frac{d+\alpha-1}{2}} \pi^{-\frac{d-1}{2}}|y|^{-\frac{d-\alpha+1}{2}}
        e^{-m^{1/\alpha}|y|}\int_{a\kappa(t,y)}^{b\kappa(t,y)} \eta(z,1)\,\dx[z],
    \]
    for $\frac{t}{|y|^{2/\alpha}}<M.$ Using this formula also for $a=0$ and $b=c$, by \eqref{eq:basicestofeta}, we get
    \eqref{eq:liminfpositive}.
    
        \end{proof}

    \begin{lemma}\label{lemm:egal}
        For all $\lambda<m$ and $a\in (0,1)$ we have
        \[
        \lim\limits_{\substack{|y|\to\infty,\,t\to\infty\\ t/|y|^{\alpha/2} \leq M }} \frac{\int\limits_0^{\kappa(t,y)} e^{(m-\lambda)\frac{|y|^{\alpha/2}}{\sqrt{2^\alpha m}}s}\eta(s,1)\, \dx[s]}{\int\limits_{a\kappa(t,y)}^{\kappa(t,y)} e^{(m-\lambda)\frac{|y|^{\alpha/2}}{\sqrt{2^\alpha m}}s}\eta(s,1)\, \dx[s]} = 1.
    \]
    \end{lemma}
    \begin{proof}
        For every fixed $a\in (0,1)$ and
    $b=\frac12(a+1),$ using \eqref{eq:basicestofeta} with $T=M\sqrt{2^\alpha m}$ we get
    \begin{align*}
        \frac{\int\limits_0^{a\kappa(t,y)} e^{(m-\lambda)\frac{|y|^{\alpha/2}}{\sqrt{2^\alpha m}}s}\eta(s,1)\, \dx[s]}{\int\limits_{a\kappa(t,y)}^{\kappa(t,y)} e^{(m-\lambda)\frac{|y|^{\alpha/2}}{\sqrt{2^\alpha m}}s}\eta(s,1)\, \dx[s]} 
        & \leq \frac{\int\limits_0^{a\kappa(t,y)} e^{(m-\lambda)\frac{|y|^{\alpha/2}}{\sqrt{2^\alpha m}}s}\eta(s,1)\, \dx[s]}{\int\limits_{b\kappa(t,y)}^{\kappa(t,y)} e^{(m-\lambda)\frac{|y|^{\alpha/2}}{\sqrt{2^\alpha m}}s}\eta(s,1)\, \dx[s]} 
         \leq e^{(m-\lambda)(a-b)t} \frac{\int\limits_0^{a\kappa(t,y)} \eta(s,1)\, \dx[s]}{\int\limits_{b\kappa(t,y)}^{\kappa(t,y)} \eta(s,1)\, \dx[s]} \\
        & \leq e^{(m-\lambda)(a-b)t} \frac{c_0^2 \int\limits_0^{a\kappa(t,y)} s\, \dx[s]}{\int\limits_{b\kappa(t,y)}^{\kappa(t,y)} s\, \dx[s] } = e^{(m-\lambda)(a-b)t} \frac{c_0^2a^2}{1-b^2}\to 0,
    \end{align*}
    as $t\to\infty$
    and the Lemma follows.
    \end{proof}

    \begin{lemma}\label{lm:etabyinteta} For every $\lambda < m$ we have 
        \begin{equation}\label{eq:limitm}
            \lim\limits_{\substack{|y|\to\infty,\,t\to\infty\\ t/|y|^{\alpha/2}<M}}
            \frac{\eta(\kappa(t,y),1)e^{(m-\lambda)t}|y|^{-\alpha/2}}{\int\limits_0^{\kappa(t,y)} e^{(m-\lambda)\frac{|y|^{\alpha/2}}{\sqrt{2^\alpha m}}s}\eta(s,1)\, \dx[s]} = \frac{m-\lambda}{\sqrt{2^\alpha m}}.
        \end{equation}
    \end{lemma}
    \begin{proof}
        For every $\varepsilon\in (0,\mathcal{C}_{\alpha/2})$ we can choose $\delta>0$ such that
        \[
            \mathcal{C}_{\alpha/2} - \varepsilon \leq \frac{\eta(s,1)}{s} \leq \mathcal{C}_{\alpha/2} + \varepsilon, \quad s\in (0,\delta].
        \] 
        and this, for $\lambda<m,$ yields 
        \begin{equation}\label{eq:limsupfirst}
        \begin{split}
           \limsup\limits_{\substack{|y|\to\infty,\,t\to\infty\\ \kappa(t,y)<\delta}}
            \frac{\eta(\kappa(t,y),1)e^{(m-\lambda)t}|y|^{-\alpha/2}}{\int\limits_0^{\kappa(t,y)} e^{(m-\lambda)\frac{|y|^{\alpha/2}}{\sqrt{2^\alpha m}}s}\eta(s,1)\, \dx[s]}
            & \leq \frac{\mathcal{C}_{\alpha/2}+\varepsilon}{\mathcal{C}_{\alpha/2}-\varepsilon} \limsup\limits_{\substack{|y|\to\infty,\,t\to\infty\\ \kappa(t,y)<\delta}}
            \frac{\kappa(t,y)e^{(m-\lambda)t}|y|^{-\alpha/2}}{\int\limits_0^{\kappa(t,y)} e^{(m-\lambda)\frac{|y|^{\alpha/2}}{\sqrt{2^\alpha m}}s}s\, \dx[s]} \\
            & =
            \frac{\mathcal{C}_{\alpha/2}+\varepsilon}{\mathcal{C}_{\alpha/2}-\varepsilon}\limsup\limits_{\substack{|y|\to\infty,\,t\to\infty\\ \kappa(t,y)<\delta}}
            \frac{e^{(m-\lambda)t}(m-\lambda)^2 t}{\sqrt{2^\alpha m}\left(e^{(m-\lambda) t}((m-\lambda)t-1) + 1\right)} \\
            & = \frac{\mathcal{C}_{\alpha/2}+\varepsilon}{\mathcal{C}_{\alpha/2}-\varepsilon} \cdot \frac{m-\lambda}{\sqrt{2^\alpha m} }
        \end{split}
        \end{equation}
    Similarly, we get
    \begin{equation}\label{eq:liminffirst}
    \begin{split}
         \liminf\limits_{\substack{|y|\to\infty,\,t\to\infty\\ \kappa(t,y)<\delta}}
            \frac{\eta(\kappa(t,y),1)e^{(m-\lambda)t}|y|^{-\alpha/2}}{\int\limits_0^{\kappa(t,y)} e^{(m-\lambda)\frac{|y|^{\alpha/2}}{\sqrt{2^\alpha m}}s}\eta(s,1)\, \dx[s]} 
            & \geq \frac{\mathcal{C}_{\alpha/2}-\varepsilon}{\mathcal{C}_{\alpha/2}+\varepsilon} \cdot \frac{m-\lambda}{\sqrt{2^\alpha m} }.
    \end{split}
    \end{equation}
    
    Let $c_1 := \inf\limits_{s\in [\frac12 \delta,M\sqrt{2^\alpha m}]}\eta(s,1).$ The function $\eta(\cdot,1)$ is uniformly continuous in $[\frac12\delta,M\sqrt{2^\alpha m}],$ so we can choose $a\in (\frac12,1)$ such that $|\eta(s,1) - \eta(\kappa(t,y),1)| < \varepsilon c_1,$ provided $s\in (a\kappa(t,y),\kappa(t,y)) \subset [\frac12 \delta,M\sqrt{2^\alpha m}],$ and we get
    \begin{equation}\label{eq:limsupsecond}
    \begin{split}
        & \limsup\limits_{\substack{|y|\to\infty,\,t\to\infty\\ \kappa(t,y)\in[\delta,M\sqrt{2^\alpha m}]}}
            \frac{\eta(\kappa(t,y),1)e^{(m-\lambda)t}|y|^{-\alpha/2}}{\int\limits_{a\kappa(t,y)}^{\kappa(t,y)} e^{(m-\lambda)\frac{|y|^{\alpha/2}}{\sqrt{2^\alpha m}}s}\eta(s,1)\, \dx[s]} \\& 
            \leq  \limsup\limits_{\substack{|y|\to\infty,\,t\to\infty\\ \kappa(t,y)\in[\delta,M\sqrt{2^\alpha m}]}}
            \frac{\eta(\kappa(t,y),1)e^{(m-\lambda)t}|y|^{-\alpha/2}}{(\eta(\kappa(t,y),1)-\varepsilon c_1)\int\limits_{a\kappa(t,y)}^{\kappa(t,y)} e^{(m-\lambda)\frac{|y|^{\alpha/2}}{\sqrt{2^\alpha m}}s}\, \dx[s]} \\
            & = \frac{m-\lambda}{\sqrt{2^\alpha m}}\limsup\limits_{\substack{|y|\to\infty,\,t\to\infty\\ \kappa(t,y)\in[\delta,M\sqrt{2^\alpha m}]}}
            \frac{\eta(\kappa(t,y),1)}{(\eta(\kappa(t,y),1)-\varepsilon c_1)(1-e^{(m-\lambda)(a-1)t})} \\
            & = \frac{m-\lambda}{\sqrt{2^\alpha m}}\limsup\limits_{\substack{|y|\to\infty,\,t\to\infty\\ \kappa(t,y)\in[\delta,M\sqrt{2^\alpha m}]}} \left(1-\frac{\varepsilon c_1}{\eta(\kappa(t,y),1)}\right)^{-1}
            \leq \frac{m-\lambda}{\sqrt{2^\alpha m}} \left(1-\varepsilon\right)^{-1}.
        \end{split}
        \end{equation}
        Similarly, we obtain
        \begin{equation}\label{eq:liminfsecond}
            \liminf\limits_{\substack{|y|\to\infty,\,t\to\infty\\ \kappa(t,y)\in[\delta,M\sqrt{2^\alpha m}]}}
            \frac{\eta(\kappa(t,y),1)e^{(m-\lambda)t}|y|^{-\alpha/2}}{\int\limits_{a\kappa(t,y)}^{\kappa(t,y)} e^{(m-\lambda)\frac{|y|^{\alpha/2}}{\sqrt{2^\alpha m}}s}\eta(s,1)\, \dx[s]} 
            \geq  \frac{m-\lambda}{\sqrt{2^\alpha m}} \left(1+\varepsilon\right)^{-1},
        \end{equation}
        and finally, from \eqref{eq:limsupfirst}, \eqref{eq:liminffirst}, \eqref{eq:limsupsecond}, \eqref{eq:liminfsecond} 
        and Lemma \ref{lemm:egal} we have
        \begin{align*}
            & \frac{m-\lambda}{\sqrt{2^\alpha m}} \min\left\{\frac{\mathcal{C}_{\alpha/2}-\varepsilon}{\mathcal{C}_{\alpha/2}+\varepsilon},(1+\varepsilon)^{-1}\right\}  \leq 
            \liminf\limits_{\substack{|y|\to\infty,\,t\to\infty\\ \kappa(t,y)\leq M\sqrt{2^\alpha m}}}
            \frac{\eta(\kappa(t,y),1)e^{(m-\lambda)t}|y|^{-\alpha/2}}{\int\limits_{0}^{\kappa(t,y)} e^{(m-\lambda)\frac{|y|^{\alpha/2}}{\sqrt{2^\alpha m}}s}\eta(s,1)\, \dx[s]} \\
            & \limsup\limits_{\substack{|y|\to\infty,\,t\to\infty\\ \kappa(t,y)\leq M\sqrt{2^\alpha m}}}
            \frac{\eta(\kappa(t,y),1)e^{(m-\lambda)t}|y|^{-\alpha/2}}{\int\limits_{0}^{\kappa(t,y)} e^{(m-\lambda)\frac{|y|^{\alpha/2}}{\sqrt{2^\alpha m}}s}\eta(s,1)\, \dx[s]} \leq 
            \frac{m-\lambda}{\sqrt{2^\alpha m}} \max\left\{\frac{\mathcal{C}_{\alpha/2}+\varepsilon}{\mathcal{C}_{\alpha/2}-\varepsilon},(1-\varepsilon)^{-1}\right\},            
        \end{align*}
        for every $\varepsilon\in (0,\mathcal{C}_{\alpha/2}),$ which yields \eqref{eq:limitm}.
    \end{proof}

   \begin{proof}[Proof of Theorem \ref{th:relatproc}, Part 3.]
       It follows from Lemma \ref{lemma:liminfpositive} that
       for all $1 > b > a > 0$ and $t>0$ we have
       \begin{align*}
          \liminf\limits_{\substack{|y|\to\infty,\,t\to\infty\\ t/|y|^{\alpha/2}<M}}  \frac{\int\limits_{0}^{t} e^{-\lambda s}p(s,y)\,\dx[s]}{\int\limits_0^{at} e^{-\lambda s}p(s,y)\,\dx[s]}
           & = \liminf\limits_{\substack{|y|\to\infty,\,t\to\infty\\ t/|y|^{\alpha/2}<M}} \frac{\int\limits_{0}^{t} e^{(m-\lambda) s} e^{-ms}p(s,y)\,\dx[s]}{\int\limits_0^{at} e^{(m-\lambda) s} e^{-ms} p(s,y)\,\dx[s]} \\
         & \geq  \liminf\limits_{\substack{|y|\to\infty,\,t\to\infty\\ t/|y|^{\alpha/2}<M}} e^{(m-\lambda)(b-a)t} \frac{\int\limits_{bt}^{t}e^{-ms}p(s,y)\,\dx[s]}{\int\limits_0^{at}  e^{-ms} p(s,y)\,\dx[s]} = \infty.
       \end{align*}
       This yields that for every fixed $a\in (0,1)$ we have
       \begin{equation}\label{eq:thesameasympt}
           \int_0^t e^{-\lambda s}p(s,y)\,\dx[s] \sim \int_{at}^t e^{-\lambda s} p(s,y)\,\dx[s].
       \end{equation}
       
       By substitutions $z=\frac{2m^{1/\alpha}}{|y|}u$, $w=\frac{1}{\sqrt{z}}-\sqrt{z}$, $s=\frac{\sqrt{2^{\alpha}m}}{|y|^{\alpha/2}}\tau$ and the homogeneity of $\eta$ we obtain
    \begin{equation}\label{eq:substitution}
        \begin{split}
            \int_{at}^t \lambda e^{-\lambda \tau} p(\tau,y)\,\dx[\tau] &=
           \frac{\lambda}{\left(4\pi\right)^{d/2}}\int_{at}^t e^{(m-\lambda) \tau} \int_0^\infty u^{-d/2} e^{-\left(\frac{|y|^2}{4u}+m^{2/\alpha}u\right)}\eta(\tau,u)\, \dx[u]\dx[\tau]\\&=
            \frac{\lambda}{(4\pi)^{d/2}} \left(\frac{|y|}{2m^{1/\alpha}}\right)^{-\frac{d-\alpha}{2}} e^{-m^{1/\alpha}|y|}
            \int_{a\kappa(t,y)}^{\kappa(t,y)} e^{(m-\lambda)\frac{|y|^{\alpha/2}}{\sqrt{2^\alpha m}}s} \int_\mathbb{R} h_s(w,|y|)\, \dx[w]\dx[s],
        \end{split}
    \end{equation}
    where
    \begin{itemize}
        \item $h_s(w,|y|)= z(w)^{-d/2} e^{-\frac{1}{2}m^{1/\alpha}|y|w^2} \eta\left(s,z(w)\right) \left(-z^\prime(w)\right)$,
        \item $z(w) = 1+\frac{w^2}{2}-\frac{w}{2}\sqrt{w^2+4}$,
        \item $z^\prime(w) = \frac{d}{dw} z(w) = w-\frac{w^2+2}{\sqrt{w^2+4}}$.
    \end{itemize}
    Note that $z^\prime(w)<0$, for every $w\in\mathbb{R}$. 
   
    Let $\varepsilon>0$ and $s\in (a\kappa(t,y),\kappa(t,y)).$ Then we have 
    \begin{align*}
            \int\limits_{|w|>\varepsilon} h_s(w,|y|) \dx[w]&\leq
            e^{-\frac{1}{2}m^{1/\alpha}|y|\varepsilon^2}\int\limits_0^\infty z^{-d/2}\eta(s,z)\dx[z]  =
            e^{-\frac{1}{2}m^{1/\alpha}|y|\varepsilon^2} s^{-d/\alpha}\int\limits_0^\infty z^{-d/2} \eta(1,z)\,\dx[z] \\
            & \leq e^{-\frac{1}{2}m^{1/\alpha}|y|\varepsilon^2} (a\kappa(t,y))^{-d/\alpha}\int\limits_0^\infty z^{-d/2} \eta(1,z)\,\dx[z]  = c_1 e^{-\frac{1}{2}m^{1/\alpha}|y|\varepsilon^2} (a\kappa(t,y))^{-d/\alpha} 
    \end{align*}
    for some constant $c_1 >0$ (the finiteness of $c_1=\int_0^\infty z^{-d/2} \eta(1,z)\,\dx[z]$ follows, e.g., from \cite[Lemma 1]{Hawkes1971}, see also \cite{GrzywnyLezajTrojan2023}).

    Moreover, we can choose $\varepsilon>0$ such that $z(w)\in\left[\frac{1}{2}, \frac{3}{2}\right]$ and $-z^\prime(w)\in\left[\frac{1}{2}, \frac{3}{2}\right]$ for all $|w|<\varepsilon$ and this yields
    \begin{align*}
            \int\limits_{|w|<\varepsilon} h_s(w,|y|) \dx[w] &=
            \int\limits_{|w|<\varepsilon} \left(z(w)\right)^{-d/2} e^{-\frac{1}{2} m^{1/\alpha} |y| w^2} \eta(s, z(w)) (-z^\prime(w))\,\dx[w] \\&\geq
            \frac12 \left(\frac32\right)^{-d/2} e^{-\frac{1}{2} m^{1/\alpha}|y| \frac{\varepsilon^2}{4}} \int\limits_{|w|<\frac{\varepsilon}{2}} \eta\left(s, z(w)\right) \dx[w] \\
            & = \frac12 \left(\frac32\right)^{-d/2} e^{-\frac{1}{8} m^{1/\alpha}|y| \varepsilon^2}\int\limits_{|w|<\frac{\varepsilon}{2}} \frac{\eta\left(sz(w)^{-\alpha/2}, 1\right)}{z(w)} \dx[w] \\
            & \geq \frac12 \left(\frac32\right)^{-d/2} e^{-\frac{1}{8} m^{1/\alpha}|y| \varepsilon^2} \int\limits_{|w|<\frac{\varepsilon}{2}} c_0^{-1} sz(w)^{-\alpha/2-1}\, \dx[w] \\
            & \geq c_2 \varepsilon e^{-\frac{1}{8} m^{1/\alpha}|y| \varepsilon^2} a\kappa(t,x),
    \end{align*}
    where we use \eqref{eq:basicestofeta} with $T=\left(\frac12\right)^{-\alpha/2}M\sqrt{2^\alpha m}=2^\alpha M\sqrt{m}.$

    In consequence
    \begin{align*}
        \lim\limits_{\substack{|y|\to\infty,t\to\infty \\t/|y|^{\alpha/2}<M}}
        \frac{\displaystyle\int\limits_{|w|>\varepsilon} h_s(w,|y|) \dx[w]}{\displaystyle \int\limits_{|w|<\varepsilon} h_s(w,|y|) \dx[w]} & \leq
        \lim\limits_{\substack{|y|\to\infty,t\to\infty \\ t/|y|^{\alpha/2}<M}}
        \frac{c_1 a^{-d/\alpha} e^{-\frac12 m^{1/\alpha} |y|\varepsilon^2}\kappa(t,y)^{-d/\alpha}}{c_2 a \varepsilon e^{-\frac18 m^{1/\alpha} |y|\varepsilon^2}\kappa(t,y)} \\
        & = 
        \lim\limits_{\substack{|y|\to\infty,\, t\to\infty \\ t/|y|^{\alpha/2}<M}}
        c_3a^{-d/\alpha-1} \varepsilon^{-1} e^{-\frac{3}{8}m^{1/\alpha}|y|\varepsilon^2} |y|^{\frac{d+\alpha}{2}} t^{-\frac{d+\alpha}{\alpha}} =
        0,
    \end{align*}
    and since the above convergence is uniform with respect to $s\in (a\kappa(t,y),\kappa(t,y)),$ this yields
     \begin{equation}\label{eq:changetoepsilon}
        \int_{a\kappa(t,y)}^{\kappa(t,y)} e^{(m-\lambda)\frac{|y|^{\alpha/2}}{\sqrt{2^\alpha m}}s} \int_\mathbb{R} h_s(w,|y|) \,\dx[w]\dx[s] \sim
        \int_{a\kappa(t,y)}^{\kappa(t,y)} e^{(m-\lambda)\frac{|y|^{\alpha/2}}{\sqrt{2^\alpha m}}s} \int_{|w|<\varepsilon} h_s(w,|y|) \,\dx[w]\dx[s].
    \end{equation}
    If $|w|<\varepsilon$ and $z(w)\in\left(\frac12, \frac32\right)$, then for every $\varepsilon_1>0$ 
    we can choose $\delta>0$, such that
    \[
            \frac{\eta(s,z(w))}{\eta(s,1)} =
            z\left(w\right)^{-\alpha/2-1}
            \frac{
                \eta\left(\dfrac{s}{z(w)^{\alpha/2}},1\right)
                }{
                \dfrac{s}{z(w)^{\alpha/2}}
                }
            \frac{
                s
                }{
                \eta\left(s,1\right)
                } \leq 
            z\left(w\right)^{-\alpha/2-1} \frac{\mathcal{C}_{\alpha/2}+\varepsilon_1}{\mathcal{C}_{\alpha/2}-\varepsilon_1},
   \]
    provided $s < \delta.$ Similarly, we get also the corresponding lower bound.
    Furthermore, the function $(s,z)\mapsto\frac{\eta(s,z)}{\eta(s,1)}$ is uniformly continuous on $[\delta,M\sqrt{2^\alpha m}]\times[\frac12,\frac32]$ and taking $\varepsilon$ sufficiently small, we get 
    \[         
          1-\varepsilon_1 \leq \frac{\eta(s,z(w))}{\eta(s,1)} \leq 1+\varepsilon_1, \quad s\in [\delta, M\sqrt{2^\alpha m}], |w|<\varepsilon.  
    \] 
    Therefore, we obtain
    \[
         j_1(\varepsilon_1,w) \leq \frac{\eta(s,z(w))}{\eta(s,1)} \leq j_2(\varepsilon_1,w), \quad s\in (0,M\sqrt{2^\alpha m}], |w|<\varepsilon, 
    \]
    where
    \[
        j_1(\varepsilon_1,w) = \min\{z(w)^{-\alpha/2-1}\frac{\mathcal{C}_{\alpha/2}-\varepsilon_1}{\mathcal{C}_{\alpha/2}+\varepsilon_1},1-\varepsilon_1\},
        \quad j_2(\varepsilon_1,w) = \max\{z(w)^{-\alpha/2-1}\frac{\mathcal{C}_{\alpha/2}+\varepsilon_1}{\mathcal{C}_{\alpha/2}-\varepsilon_1},1+\varepsilon_1\}.
    \]
    We obtain
    \begin{equation}\label{eq:intgeqnupm}
         \frac{\int\limits_{|w|<\varepsilon} h_s(w,|y|)\,  \dx[w]}{\eta(s,1)} \leq
        \int\limits_{|w|<\varepsilon} j_2(\varepsilon_1,w) z(w)^{-\frac{d}{2}} e^{-\frac{1}{2}m^{1/\alpha}|y|w^2}\left(-z^\prime(w)\right)\dx[w].
    \end{equation}
    Using the Laplace's method we get
    \[
       \lim_{\substack{|y|\to\infty\\ t\to\infty }} \frac{\int\limits_{|w|<\varepsilon} j_2(\varepsilon_1,w) z(w)^{-\frac{d}{2}} e^{-\frac{1}{2}m^{1/\alpha}|y|w^2}\left(-z^\prime(w)\right)\dx[w]}{\sqrt{\frac{2\pi}{|y|m^{1/\alpha}}}} = 
       \max\left\{\frac{\mathcal{C}_{\alpha/2}+\varepsilon_1}{\mathcal{C}_{\alpha/2}-\varepsilon_1},1+\varepsilon_1\right\}
    \]
    and this, \eqref{eq:changetoepsilon}, \eqref{eq:substitution}, \eqref{eq:thesameasympt}, \eqref{eq:intgeqnupm} and
    Lemma \ref{lemm:egal} yield
    \begin{equation*}\label{eq:ineq_relat_U}
        \limsup_{\substack{|y|\to\infty, t\to\infty\\ t/|y|^{\alpha/2}<M}}
        \frac{\displaystyle \int_{0}^{\kappa(t,y)} e^{(m-\lambda)\frac{|y|^{\alpha/2}}{\sqrt{2^\alpha m}}s} \int_{\R} h_s(w,|y|) \dx[w]\dx[s]}
        {\displaystyle \int_{0}^{\kappa(t,y)} e^{(m-\lambda)\frac{|y|^{\alpha/2}}{\sqrt{2^\alpha m}}s} \eta(s,1) \dx[s]\sqrt{\dfrac{2\pi}{|y|m^{1/\alpha}}}} \leq 
        \max\left\{\frac{\mathcal{C}_{\alpha/2}+\varepsilon_1}{\mathcal{C}_{\alpha/2}-\varepsilon_1},1+\varepsilon_1\right\}.
    \end{equation*}   
    Similarly, we get
    \begin{equation*}\label{eq:ineq_relat_L}
        \liminf_{\substack{|y|\to\infty, t\to\infty\\ t/|y|^{\alpha/2}<M }}
        \frac{\displaystyle \int_{0}^{\kappa(t,y)} e^{(m-\lambda)\frac{|y|^{\alpha/2}}{\sqrt{2^\alpha m}}s} \int_{\R} h_s(w,|y|) \dx[w]\dx[s]}
        {\displaystyle \int_{0}^{\kappa(t,y)} e^{(m-\lambda)\frac{|y|^{\alpha/2}}{\sqrt{2^\alpha m}}s} \eta(s,1) \dx[s]\sqrt{\dfrac{2\pi}{|y|m^{1/\alpha}}}} \geq 
        \min\left\{\frac{\mathcal{C}_{\alpha/2}-\varepsilon_1}{\mathcal{C}_{\alpha/2}+\varepsilon_1},1-\varepsilon_1\right\},
    \end{equation*}
    and since the both above inequalities hold for all $\varepsilon_1 < \min\{1,\mathcal{C}_{\alpha/2}\},$ we have
    \[
        \lim_{\substack{|y|\to\infty, t\to\infty\\ t/|y|^{\alpha/2}<M }}
        \frac{\displaystyle \int_{0}^{\kappa(t,y)} e^{(m-\lambda)\frac{|y|^{\alpha/2}}{\sqrt{2^\alpha m}}s} \int_{\R} h_s(w,|y|)\, \dx[w]\dx[s]}
        {\displaystyle \int_{0}^{\kappa(t,y)} e^{(m-\lambda)\frac{|y|^{\alpha/2}}{\sqrt{2^\alpha m}}s} \eta(s,1) \,\dx[s]\,\sqrt{\dfrac{2\pi}{|y|m^{1/\alpha}}}} = 1.
    \]
    Finally, we have
    \begin{equation}\label{eq:truncrespart}
        \int_0^t \lambda e^{-\lambda s} p(s,y)\, \dx[s] \sim 
        \frac{\lambda m^{\frac{d-\alpha-1}{2\alpha}}e^{-m^{1/\alpha}|y|}}{2^{\frac{d+\alpha-1}{2}}\pi^{\frac{d-1}{2}}|y|^{\frac{d-\alpha+1}{2}}}
        \int_0^{\frac{\sqrt{2^\alpha m}}{|y|^{\alpha/2}} t} e^{(m-\lambda)\frac{|y|^{\alpha/2}}{\sqrt{2^\alpha m}}s} \eta(s,1)\, \dx[s].
    \end{equation}
    In order to find the asymptotic behavior of $p_\lambda(t,x,y)$, we consider also $e^{-\lambda t}p(t,x,y)$. 
    Using Corollary \ref{reldensprop} and Lemma \ref{lm:etabyinteta} we obtain
    \begin{align*}
        & \lim\limits_{\substack{|y|\to\infty,\, t\to\infty\\t/|y|^{\alpha/2}<M}}
            \frac{e^{-\lambda t} p(t,y-x) }{\displaystyle  |y|^{-\frac{d-\alpha+1}{2}} e^{-m^{1/\alpha}(|y|-\langle \frac{y}{|y|},x\rangle)} \frac{\eta(\kappa(t,y-x),1)}{\eta(\kappa(t,y),1)}\int_0^{\kappa(t,y)} e^{(m-\lambda)\frac{|y|^{\alpha/2}}{\sqrt{2^\alpha m}}s} \eta(s,1) \,\dx[s]} \\
            & =
            \lim\limits_{\substack{|y|\to\infty,\, t\to\infty\\t/|y|^{\alpha/2}<M}}
            \frac{e^{(m-\lambda) t} \left(\frac{m^{1/\alpha}}{2\pi}\right)^\frac{d-1}{2} e^{m^{1/\alpha}(|y|-|y-x|-\langle \frac{y}{|y|},x\rangle)}|y-x|^{-\frac{d+1}{2}}|y|^{\frac{d-\alpha+1}{2}}\eta(\kappa(t,y),1)}{\displaystyle   \int_0^{\kappa(t,y)} e^{(m-\lambda)\frac{|y|^{\alpha/2}}{\sqrt{2^\alpha m}}s} \eta(s,1) \,\dx[s]} \\
            & = \left(\frac{m^{1/\alpha}}{2\pi}\right)^\frac{d-1}{2} \frac{m-\lambda}{\sqrt{2^\alpha m}},
    \end{align*}
    and the theorem follows by this and \eqref{eq:truncrespart}.  
    \end{proof}

%\bibliographystyle{abbrv}
%\bibliography{bibliografia.bib} 

\begin{thebibliography}{10}

\bibitem{BlumenthalGetoor1960}
R.~M. Blumenthal and R.~K. Getoor.
\newblock Some theorems on stable processes.
\newblock {\em Trans. Amer. Math. Soc.}, 95:263--273, 1960.

\bibitem{Bogdan}
K.~Bogdan, T.~Grzywny, and M.~Ryznar.
\newblock Density and tails of unimodal convolution semigroups.
\newblock {\em J. Funct. Anal.}, 266(6):3543--3571, 2014.

\bibitem{Carmona1989}
R.~Carmona.
\newblock Path integrals for relativistic {S}chr\"odinger operators.
\newblock In {\em Schr\"odinger operators ({S}\o nderborg, 1988)}, volume 345
  of {\em Lecture Notes in Phys.}, pages 65--92. Springer, Berlin, 1989.

\bibitem{Cygan}
W.~Cygan, T.~Grzywny, and B.~Trojan.
\newblock Asymptotic behavior of densities of unimodal convolution semigroups.
\newblock {\em Trans. Amer. Math. Soc.}, 369(8):5623--5644, 2017.

\bibitem{dibello2025}
C.~Di~Bello, A.~Chechkin, T.~Grzywny, Z.~Palmowski, K.~Szczypkowski, and
  B.~Trojan.
\newblock Partial versus total resetting for {L}\'evy flights in {$d$}
  dimensions: similarities and discrepancies.
\newblock {\em Chaos}, 35(4):Paper No. 043129, 9, 2025.

\bibitem{Dumas2002}
V.~Dumas, F.~Guillemin, and P.~Robert.
\newblock A {M}arkovian analysis of additive-increase multiplicative-decrease
  algorithms.
\newblock {\em Adv. in Appl. Probab.}, 34(1):85--111, 2002.

\bibitem{EvansMajumdar2011}
M.~R. Evans and S.~N. Majumdar.
\newblock Diffusion with stochastic resetting.
\newblock {\em Phys. Rev. Lett.}, 106:160601, Apr 2011.

\bibitem{Evans2020}
M.~R. Evans, S.~N. Majumdar, and G.~Schehr.
\newblock Stochastic resetting and applications.
\newblock {\em Journal of Physics A: Mathematical and Theoretical},
  53(19):193001, Apr 2020.

\bibitem{Froehlich2007}
J.~Fr\"ohlich, B.~L.~G. Jonsson, and E.~Lenzmann.
\newblock Boson stars as solitary waves.
\newblock {\em Comm. Math. Phys.}, 274(1):1--30, 2007.

\bibitem{GrzywnyLezajTrojan2023}
T.~Grzywny, {\L}.~Le\.zaj, and B.~Trojan.
\newblock Transition densities of subordinators of positive order.
\newblock {\em J. Inst. Math. Jussieu}, 22(3):1119--1179, 2023.

\bibitem{Grzywny2024}
T.~Grzywny, Z.~Palmowski, K.~Szczypkowski, and B.~Trojan.
\newblock Stationary states for stable processes with partial resetting.
\newblock {\em Ann. Appl. Probab.}, 36(2):1110--1177, 2026.

\bibitem{Guillemin2002}
F.~Guillemin, P.~Robert, and A.~Zwart.
\newblock {AIMD algorithms and exponential functionals}.
\newblock Research Report RR-4447, {INRIA}, 2002.

\bibitem{Gupta2022}
S.~Gupta and A.~M. Jayannavar.
\newblock Stochastic resetting: A (very) brief review.
\newblock {\em Frontiers in Physics}, 10, 2022.

\bibitem{Hawkes1971}
J.~Hawkes.
\newblock A lower {L}ipschitz condition for the stable subordinator.
\newblock {\em Z. Wahrscheinlichkeitstheorie und Verw. Gebiete}, 17:23--32,
  1971.

\bibitem{Herr2014}
S.~Herr and E.~Lenzmann.
\newblock The {B}oson star equation with initial data of low regularity.
\newblock {\em Nonlinear Anal.}, 97:125--137, 2014.

\bibitem{KaletaSchillingSztonyk2025}
K.~Kaleta, R.~L. Schilling, and P.~Sztonyk.
\newblock Decay of resolvent kernels and {S}chr\"odinger eigenstates for
  {L}\'evy operators.
\newblock {\em Math. Ann.}, 394(4):Paper No. 88, 36, 2026.

\bibitem{KaletaSztonyk2019}
K.~Kaleta and P.~Sztonyk.
\newblock Spatial asymptotics at infinity for heat kernels of
  integro-differential operators.
\newblock {\em Trans. Amer. Math. Soc.}, 371(9):6627--6663, 2019.

\bibitem{kolb}
M.~Kolb and A.~W\"ubker.
\newblock Brownian motion with partial resetting conditioned to stay positive.
\newblock {\em Bull. Lond. Math. Soc.}, 58(3):Paper No. e70314, 24, 2026.

\bibitem{Kou2002}
S.~G. Kou.
\newblock A jump-diffusion model for option pricing.
\newblock {\em Management Science}, 48(8):1086--1101, 2002.

\bibitem{Lieb2010}
E.~H. Lieb and R.~Seiringer.
\newblock {\em The stability of matter in quantum mechanics}.
\newblock Cambridge University Press, Cambridge, 2010.

\bibitem{Mendoza2010}
E.~G. Mendoza.
\newblock Sudden stops, financial crises, and leverage.
\newblock {\em American Economic Review}, 100(5):1941–66, December 2010.

\bibitem{MERTON1976}
R.~C. Merton.
\newblock Option pricing when underlying stock returns are discontinuous.
\newblock {\em Journal of Financial Economics}, 3(1):125--144, 1976.

\bibitem{MontanariZecchina2002}
A.~Montanari and R.~Zecchina.
\newblock Optimizing searches via rare events.
\newblock {\em Phys. Rev. Lett.}, 88:178701, Apr 2002.

\bibitem{Nemes2013}
G.~Nemes.
\newblock An explicit formula for the coefficients in {L}aplace's method.
\newblock {\em Constr. Approx.}, 38(3):471--487, 2013.

\bibitem{NIST}
F.~W.~J. Olver et~al. (eds.).
\newblock {\it NIST Digital Library of Mathematical Functions}.
\newblock http://dlmf.nist.gov/, Release 1.1.2 of 2021-06-15.

\bibitem{RoldanLisicaGrill2016}
E.~Rold\'an, A.~Lisica, D.~S\'anchez-Taltavull, and S.~W. Grill.
\newblock Stochastic resetting in backtrack recovery by rna polymerases.
\newblock {\em Phys. Rev. E}, 93:062411, Jun 2016.

\bibitem{Ryznar2002}
M.~Ryznar.
\newblock Estimates of {G}reen function for relativistic {$\alpha$}-stable
  process.
\newblock {\em Potential Anal.}, 17(1):1--23, 2002.

\bibitem{Sato68}
K.~Sato.
\newblock {\em L\'evy processes and infinitely divisible distributions.}
  Cambridge Studies in Advanced Mathematics, vol. 68.
\newblock Cambridge University Press, Cambridge, revised edition, 2013.

\bibitem{SchillingPartzsch2014}
R.~L. Schilling and L.~Partzsch.
\newblock {\em Brownian motion}.
\newblock De Gruyter Graduate. De Gruyter, Berlin, second edition, 2014.

\bibitem{Zolotariev1999}
V.~V. Uchaikin and V.~M. Zolotarev.
\newblock {\em Chance and stability}.
\newblock Modern Probability and Statistics. VSP, Utrecht, 1999.

\bibitem{watanabe}
T.~Watanabe.
\newblock Asymptotic estimates of multi-dimensional stable densities and their
  applications.
\newblock {\em Trans. Amer. Math. Soc.}, 359(6):2851--2879, 2007.

\end{thebibliography}

\bibliographystyle{abbrv}

\end{document}